\documentclass{amsart}
\usepackage{amsmath}
\usepackage{amsthm}
\usepackage{amssymb}
\usepackage{xypic}

\theoremstyle{plain}
\newtheorem{theorem}{Theorem}[section]
\newtheorem{proposition}[theorem]{Proposition}
\newtheorem{corollary}[theorem]{Corollary}
\newtheorem{lemma}[theorem]{Lemma}
\newtheorem{claim}{Claim}

\theoremstyle{definition}
\newtheorem{definition}[theorem]{Definition}

\newtheorem{construction}[theorem]{Construction}

\theoremstyle{remark}
\newtheorem{remark}[theorem]{Remark}
\newtheorem{example}[theorem]{Example}

\newcommand{\bC}{\mathbb{C}}

\newcommand{\fg}{\mathfrak{g}}
\newcommand{\fh}{\mathfrak{h}}

\newcommand{\fm}{\mathfrak{m}}

\newcommand{\ext}[1]{{#1^\dagger}}
\newcommand{\extc}[1]{\tilde{#1}}

\title[Modules over quandle spaces and Lie-Yamaguti algebras]
{Modules over quandle spaces and representations of Lie-Yamaguti algebras}

\author{Nobuyoshi Takahashi}
\address{
Department of Mathematics, Graduate School of Science, Hiroshima University, 
1-3-1 Kagamiyama, Higashi-Hiroshima, 739-8526 JAPAN}
\email{tkhsnbys@hiroshima-u.ac.jp}

\subjclass[2010]{Primary 22A30; Secondary 14M17; 17D99; 22F30}

\keywords{Quandles; Regular $s$-manifold; Lie-Yamaguti algebra; Lie triple system; Representation}

\begin{document}

\begin{abstract}
We study quandle modules 
over quandle spaces $Q$, i.e. quandles endowed with geometric structures. 
In the case $Q$ is a regular $s$-manifold, we exhibit 
how modules over $Q$  are related with representations of Lie-Yamaguti algebras. 
\end{abstract}

\maketitle

\section{Introduction}

The notion of a quandle was introduced by Joyce and Matveev 
as a tool for defining invariants of knots and links 
(\cite{Joyce1982}, \cite{Matveev1982}). 
Further invariants were supplied 
by the use of quandle cohomology theory (\cite{CJKLS2003}). 
While the quandle cohomology with the constant coefficient group is already useful, 
the natural framework for the cohomology theory is given 
by quandle modules(\cite{AG2003} and \cite{Jackson2005}). 
In addition to their usefulness in such applications, 
it is expectted that modules are important in understanding a quandle, 
just as in group theory and ring theory, and so on. 

Topological quandles(\cite{Rubinsztein2007}), 
smooth quandles(\cite{Ishikawa2018}) 
and quandle varieties(\cite{Takahashi2016}) 
were defined as quandles with the structure of 
a topological space, smooth manifold, and algebraic variety, respectively. 
We will refer to them as quandle spaces. 
As we will explain later, 
quandle spaces can also be regarded as generalizations of symmetric spaces. 

For a topological quandle $Q$, 
the notion of a module over $Q$ was defined in \cite{EM2016}, 
and it readily generalizes to other quandle spaces(Definition \ref{def_modules_sp}). 
The objective of this paper is to study modules over a quandle space $Q$, 
and especially, 
give an infinitesimal description of certain modules 
in the case $Q$ is a regular $s$-manifold, 
which is a special, but in a sense ``generic,'' kind of quandle space.

\subsection{Quandles, $\Phi$-spaces, regular $s$-manifolds and quandle spaces}

A quandle is a set $Q$ endowed with a binary operation $\rhd$ 
satisfying ``Quandle Axioms''(see Definition \ref{def_quandle}). 
Roughly, the axioms say that the map $s_x: Q\to Q; y\mapsto x\rhd y$ 
is a kind of symmetry around $x$, although it is not required to be involutive. 

It has already been noted by Joyce that 
symmetric spaces can be regarded as quandles, 
by taking $s_x$ to be the geodesic involution. 
In fact, before the quandles attracted attention from the knot theory, 
Loos (\cite{Loos1967}, \cite{Loos1969}) 
reinterpreted a symmetric space as a manifold with a binary operation, 
and this lead to the theories of regular $s$-manifolds and $\Phi$-spaces, 
which are generalizations of symmetric spaces (see \cite{Kowalski1980}, \cite{Fedenko1977}). 

A regular $s$-manifold is, in our terminology, 
a smooth quandle for which $1-d_xs_x\in\mathrm{End}(T_xQ)$ is invertible for any $x\in Q$. 
Among the most important theoretical results on regular $s$-manifolds $Q$ 
is the fact that 
$Q$ can be described as a homogeneous space $G/H$, 
and that the operation $\rhd$ comes from an automorphism $\Phi$ of $G$ 
(\cite{Kowalski1980}, \cite{Fedenko1977}). 
Such a triplet $(G, H, \Phi)$ or $G/H$ is called a $\Phi$-space.

\subsection{Modules over quandles}

Just as in group theory and ring theory, 
it seems important to consider modules over quandles. 
As an example of their use, a module over a quandle $Q$ 
corresponds to a certain extension of $Q$, 
providing a building block for constructing and classifying quandles. 
Furthermore, quandle modules provide a natural setting 
for the cohomology theory. 

The general notion of a quandle module 
was introduced and studied in \cite{AG2003} and \cite{Jackson2005} 
(Definition \ref{def_modules}). 
One feature of the definition of \cite{Jackson2005} is 
that we consider a disjoint union $\coprod_{x\in Q}A_x$ 
of different modules $A_x$ for different $x\in Q$, 
while the case of \cite{AG2003} 
(and the more naive definition as modules over the associated group)
corresponds to the case of the direct product $Q\times A$.

For a topological quandle $Q$, 
the notion of a module over $Q$ was given in \cite{EM2016}. 
We can easily generalize this definition to smooth, complex analytic and algebraic settings 
(Definition \ref{def_modules_sp}). 
One variant of a module is given by a certain space $\mathcal{A}=\coprod_{x\in Q}A_x$ 
where $A_x$ are additive groups. 
If $A_x$ are vector spaces and the module structure respects the linear structure, 
we call $\mathcal{A}$ a linear quandle module. 
In the algebraic case, there is another variant where we take $\mathcal{A}$ to be a coherent sheaf. 
The fibers $A_x$ do not have to be isomorphic to each other in general, 
but in the transitive case (see Definition \ref{def_quandle_auto}) they are, 
and we can think of $\mathcal{A}$ as a fiber bundles over $Q$.

Later in this paper, we will focus on connected regular $s$-manifolds $Q$. 
Our main result gives a description of certain linear quandle modules over $Q$ 
in terms of infinitesimal data.

\subsection{Lie-Yamaguti algebras, infinitesimal $s$-manifolds and their representations}

For a regular $s$-manifold $Q$ and $q\in Q$, 
we can endow $T=T_qQ$ with the structure 
of an ``infinitesimal $s$-manifold,''
just as the tangent space of a Lie group at the identity is a Lie algebra. 

Let us describe the correspondence in more details. 
As we stated above, 
a regular $s$-manifold $Q$ can be considered as a homogeneous space $G/H$, 
and this is a reductive homogeneous space, 
defined and studied by Nomizu(\cite{Nomizu1954}). 
Given a reductive homogeneous space, 
Nomizu considered certain tensors on its tangent space at a point. 
According to a reformulation by Yamaguti(\cite{Yamaguti1958}), 
$T$ is equipped with 
a binary operation $*$ and a ternary operation $[\ ]$ 
which make $T$ a ``general Lie triple system,'' 
or a Lie-Yamaguti algebra according to the terminology of \cite{KW01}. 

The binary operation on $Q$ furthermore defines an automorphism $\sigma$ of $T$ 
satisfying certain conditions, 
and $(T, \sigma)$ is what \cite{Kowalski1980} calls an infinitesimal $s$-manifold. 
It has been known that local isomorphism classes of regular $s$-manifolds 
are in one-to-one correspondence 
with isomorphism classes of infinitesimal $s$-manifolds 
(\cite[Chapter III]{Kowalski1980}, see Theorem \ref{thm_reg_s_mfd}).

In \cite{Yamaguti1969}, 
Yamaguti studied representations of Lie-Yamaguti algebras 
and their cohomology groups. 
A representation of a Lie-Yamaguti algebra $T$ 
is given by a quadruplet $(V, \rho, \delta, \theta)$ 
of a vector space $V$, a linear map $\rho: T\to\mathrm{End}(V)$ 
and bilinear maps $\delta, \theta: T\times T\to\mathrm{End}(V)$ 
(see Definition \ref{def_rly}).

\subsection{Main result}

In this paper, we define a representation of an infinitesimal $s$-manifold $(T, \sigma)$ 
to be a pair $(V, \psi)$ where $V$ is a representation of $T$ 
and $\psi\in\mathrm{GL}(V)$ is an invertible linear map 
satisfying certain conditions (Definition \ref{def_rism}). 

Then our main theorem is as follows. 
In the following, $Q$ can be a smooth manifold, a complex analytic manifold 
or an algebraic variety over $\mathbb{C}$. 
A regular module over $Q$ is a linear quandle module 
satisfying a ``nondegeneracy'' condition (see Definition \ref{def_regular_modules}), 
and homomorphisms are bundle maps compatible with the module structures. 

\begin{theorem}(=Theorem \ref{thm_main})
Let $Q$ be a connected regular $s$-manifold and $q\in Q$ a point. 
Write $\mathbb{K}$ for the base field. 

(1)
Given a regular quandle module $\mathcal{A}$ over $Q$, 
there is a natural structure of a regular representation of $(T_qQ, d_qs_q)$ on $\mathcal{A}_q$. 

(2)
For regular quandle modules $\mathcal{A}$ and $\mathcal{A}'$ over $Q$ 
and a homomorphism $F: \mathcal{A}\to\mathcal{A}'$, 
$F_q$ is a homomorphism of representations of $(T_qQ, d_qs_q)$. 
This correspondence gives an injective map 
between the sets of homomorphisms. 
In the smooth or complex analytic case, 
this map is surjective if $Q$ is simply-connected. 

(3)
Assume that $Q$ is a connected, simply-connected smooth or complex analytic 
regular $s$-manifold. 
Given a regular representation $V$ of $(T_qQ, d_qs_q)$, 
there exists a regular quandle module $\mathcal{A}$ over $Q$ 
such that $\mathcal{A}_q$ is isomorphic to $V$. 

\smallbreak
In other words, there is a faithful functor between the following categories: 
\begin{itemize}
\item
The category $\mathbf{Mod}_{\mathbb{K}}^r(Q)$ 
of regular quandle modules $(\mathcal{A}, \eta, \tau)$ over $(Q, \rhd)$. 
\item
The category $\mathbf{Rep}^r(T_qQ, d_qs_q)$ 
of regular representations $(V, \rho, \delta, \theta, \psi)$ 
of the infinitesimal $s$-manifold $(T_qQ, *, [\ ], d_qs_q)$. 
\end{itemize}
In the smooth or complex analytic case, 
if $Q$ is simply-connected, it is an equivalence. 
\end{theorem}

There are a couple of ingredients in the proof of this theorem. 

We are concerned with two categories, 
one of regular $s$-manifolds and the other of infinitesimal $s$-manifolds. 
In each of them, 
we have a correspondence between linear modules or representations 
and certain ``abelian group objects.'' 
Thus a functor from the category of (pointed) regular $s$-manifolds 
to that of infinitesimal $s$-manifolds 
would associate a representation to a module. 
In fact, the assignment $\mathcal{T}: Q\mapsto T_qQ$ extends to a functor 
(Corollary \ref{cor_functor}), 
i.e. a quandle homomorphism induces a homomorphism of infinitesimal $s$-manifolds 
(Theorem \ref{thm_hom}), 
but the proof of this fact is not so straightforward. 
It might be possible to prove this by combining results from \cite{Fedenko1977}, 
but we will give a proof in this paper 
since the author could not fully understand the arguments there 
and \cite{Fedenko1977} is not easily available. 

The functor $\mathcal{T}$ is essentially surjective and faithful, 
and its restriction to connected and simply-connected regular $s$-manifolds 
is an equivalence (Corollary \ref{cor_functor}). 
Thus, if $Q$ is simply-connected and $V$ is a representation of $T$, 
we can find a simply-connected regular $s$-manifold $\mathcal{A}$ over $Q$. 
We can show that the fibers of $\mathcal{A}\to Q$ are isomorphic to $V$ as manifolds. 
It follows that $\mathcal{A}\times_Q\mathcal{A}$ is also simply-connected, 
and we have an addition map $\mathcal{A}\times_Q\mathcal{A}\to \mathcal{A}$. 
Together with the scalar multiplication maps, 
this makes $\mathcal{A}$ a vector bundle over $Q$ with a quandle structure.

\subsection{Plan of the paper}

In Section 2, 
after a short account on the definition of racks and quandles and some related notions and facts, 
we define rack and quandle spaces, 
and regular $s$-manifolds in particular. 
Then we describe how regular $s$-manifolds and their homomorphisms are related with 
$\Phi$-spaces, Lie-Yamaguti algebras and infinitesimal $s$-manifolds 
and their homomorphisms. 

In Section 3, 
we recall from \cite{Jackson2005} the notion of a rack or quandle module over a rack or quandle $X$ 
and its relation with an abelian group object in the category of quandles over $X$. 
Then we introduce rack and quandle modules over a rack or quandle space. 
A couple of interesting examples of modules will also be given. 
In Section 4, 
a corresponding theory for representations is developed. 

The main theorem will be stated and proven in Section 5.

\section{Quandle spaces, Lie-Yamaguti algebras and their homomorphisms}

We recall basic definitions related to quandles(\cite{Joyce1982}). 
In this paper, we use the convention where $X$ or $Q$ ``acts'' on itself from the left. 

\begin{definition}\label{def_quandle}
A \emph{rack} is a set $X$ endowed with a binary operation $\rhd$ 
satisfying the following conditions. 
\begin{enumerate}
\item
For any $x, y\in X$, 
there exists a unique element $y'\in X$ such that $x\rhd y'=y$. 
\item
For any $x, y, z\in X$, 
$x\rhd(y\rhd z)=(x\rhd y)\rhd(x\rhd z)$. 
\end{enumerate}
We define $s_x: X\to X$ by $s_x(y)=x\rhd y$, 
which is bijective by (1). 
We write $x\rhd^{-1} y$ for $s_x^{-1}(y)$. 

We call $(X, \rhd)$ a \emph{quandle} if it additionally satisfies the following condition: 
\begin{enumerate}
\item[(3)]
For any $x\in X$, $x\rhd x=x$ holds. 
\end{enumerate}
\end{definition}

\begin{definition}
Let $X$ and $X'$ be racks or quandles. 
A \emph{homomorphism} from $X$ to $X'$ is a map $F: X\to X'$ 
such that $F(x\rhd y)=F(x)\rhd F(y)$ holds for any $x, y\in X$. 

An \emph{isomorphism} is a homomorphism which admits an inverse homomorphism. 
An isomorphism onto itself is called an \emph{automorphism}. 
\end{definition}

The following is immediate from the definitions. 
\begin{proposition}
(1) 
If $F: X\to X'$ is a homomorphism of racks, 
$F(x\rhd^{-1}y)=F(x)\rhd^{-1}F(y)$ holds. 

(2)
A homomorphism is an isomorphism if and only if it is bijective. 

(3)
For any $x\in X$, the map $s_x$ is an automorphism of $X$. 
\end{proposition}

\begin{definition}\label{def_quandle_auto}
Let $X$ be a rack. 

(1)
The \emph{automorphism group} $\mathrm{Aut}_\rhd (X)$ of $X$ 
is the set of all automorphisms of $X$. 

(2)
The \emph{group of inner automorphisms} of $X$ is the subgroup 
\[
\mathrm{Inn} (X):=\langle s_x\mid x\in X\rangle
\]
of $\mathrm{Aut}_\rhd(X)$. 

(3)
The \emph{group of transvections} of $X$ is the subgroup 
\[
\mathrm{Tr} (X):=\{s_{x_1}^{e_1}\cdots s_{x_k}^{e_k} \mid x_i\in X, 
\sum e_i=0\}
\]
of $\mathrm{Inn}_\rhd(X)$. 

(4)
A rack $X$ is called \emph{transitive}
if the action of $\mathrm{Inn} (X)$ on $X$ is transitive. 
It is often called ``connected.'' 
\end{definition}

\begin{remark}
(1)
The condition for $F: X\to X$ to be a homomorphism 
can be expressed as $F s_x=s_{F(x)}F$ for every $x\in X$.  
In particular, for $g\in\mathrm{Aut}_\rhd(X)$ 
we have $g s_x g^{-1}=s_{g(x)}$, 
and we see that $\mathrm{Inn}(X)$ and $\mathrm{Tr}(X)$ are 
normal subgroups of $\mathrm{Aut}_\rhd(X)$. 

(2)
We have in particular 
$s_xs_y=s_{x\rhd y}s_x$ for $x, y\in X$, 
and it follows that $\mathrm{Tr}(X)$ can also be described as the subgroup 
generated by, say, $\{s_xs_{x_0}^{-1}\mid x\in X\}$ for a fixed $x_0$. 

(3)
If $X$ is a quandle, 
then $X$ is transitive if and only if 
the action of $\mathrm{Tr}(X)$ on $X$ is transitive: 
If an element of $X$ is written as $s_{x_1}^{e_1}\cdots s_{x_k}^{e_k}(x_0)$, 
then it can also be expressed 
as $s_{x_1}^{e_1}\cdots s_{x_k}^{e_k}s_{x_0}^{-\sum e_i}(x_0)$. 

This does not hold for a rack in general, as seen from the case 
where $X=\mathbb{Z}$ and $x\rhd y=y+1$. 
\end{remark}

There is a description of transitive quandles 
in terms of groups and their automorphisms. 

\begin{definition}\label{def_conj}
For an element $h$ of a group $G$ 
and a normal subgroup $N$ of $G$, 
denote by $\mathrm{conj}_N(h)$ the automorphism of 
$N$ defined by $g\mapsto hgh^{-1}$. 
\end{definition}

\begin{proposition}[Theorem 7.1 of \cite{Joyce1982}]
\label{prop_conn_quandle}
(1)
If $G$ is group, $\Phi$ is an automorphism of $G$ 
and $H$ is contained in the subgroup $G^\Phi$ consisting of elements fixed by $\Phi$, 
then the operation $xH\rhd_\Phi yH:=x\Phi(x^{-1}y)H$  on $G/H$ 
is well-defined and satisfies the quandle axioms. 

(2)
For a transitive quandle $Q$ and $q\in Q$, 
let $G$ be $\mathrm{Aut}_\rhd(Q)$, $\mathrm{Inn}(Q)$ or $\mathrm{Tr}(Q)$ 
and $\Phi=\mathrm{conj}_G(s_q)$. 

Then the stabilizer $G_q$ of $q$ 
is contained in $G^\Phi$, 
and the natural map 
\[
G/G_q\to Q; \quad gG_q\mapsto gq 
\]
gives an isomorphism of quandles 
$(G/G_q, \rhd_\Phi)$ and $Q$. 
\end{proposition}

Now let us consider quandles equipped with geometric structures. 
In the setting of algebraic geometry, 
a variety will mean a reduced separated algebraic scheme 
over an algebraically closed field $k$. 

\begin{definition}
A \emph{topological quandle} (resp. a \emph{smooth quandle}, 
a \emph{complex analytic quandle} or a \emph{quandle variety} over $k$) 
is a topological space (resp. a smooth manifold, a complex analytic space 
or  an algebraic variety over $k$) $Q$ 
equipped with a quandle operation $\rhd$ 
such that $Q\times Q\to Q\times Q; (q, r)\mapsto (q, q\rhd r)$ 
is a homeomorphism (resp. a diffeomorphism, a biholomorphic map 
or a biregular map). 
\end{definition}

For topologically connected and transitive quandles, 
we have the following description. 
The smooth case is due to \cite{Ishikawa2018} 
and the algebraic case to \cite{Takahashi2016}. 
The complex analytic case is similar to the smooth case. 

\begin{theorem}
Let $Q$ be a connected and transitive smooth quandle 
(resp. complex analytic quandle 
or quandle variety over an algebraically closed field of characteristic $0$). 

Then $\mathrm{Inn}(Q)$ has a natural structure of a Lie group 
(resp. a complex Lie group or 
an algebraic group with possibly countably many connected components) 
and $\mathrm{Tr}(Q)$ is the connected component of $\mathrm{Inn}(Q)$ 
at the identity element. 

For $G=\mathrm{Inn}(Q)$ or $\mathrm{Tr}(Q)$, 
the action of $G$ on $Q$ is of class $C^\infty$ (resp. holomorphic or regular). 
Consequently, 
for any choice of $q\in Q$, 
there is a natural isomorphism 
between $(G/G_q, \rhd_{\mathrm{conj}_G(s_q)})$ and $Q$ 
of smooth quandles (resp. complex analytic quandles or quandle varieties). 
\end{theorem}

For a special kind of smooth quandles, called regular $s$-manifolds, 
many results including the above one have been known since 1970's. 

\begin{definition}(\cite[Definition II.2]{Kowalski1980})
A \emph{regular $s$-manifold} is a smooth quandle $Q$ 
such that $id_{T_qQ}-d_qs_q\in\mathrm{End}(T_qQ)$ is invertible 
for any $q\in Q$. 
\end{definition}

This was based on a reinterpretation of the notion of a symmetric space by Loos 
as a manifold with a binary operation. 

It has also been known that the tangent space of a regular $s$-manifold $Q$ 
at a point can be endowed with the structure of a certain kind of algebra, 
and that $Q$ is locally determined by this algebra.

\begin{definition}
(\cite[\S1]{Yamaguti1969})
(1)
A \emph{Lie-Yamaguti algebra} over a field $k$ is a triplet $(T, *, [\ ])$, 
where $T$ is a finite dimensional $k$-vector space, 
$*: T\times T\to T$ is a bilinear operation, 
and $[\ ]: T\times T\times T\to T$ is a trilinear operation, 
such that the following hold for any $x, y, z, v, w\in T$. 
\begin{enumerate}
\item[(LY1)]
$x*x=0$. 
\item[(LY2)]
$[x, x, y]=0$. 
\item[(LY3)]
$[x, y, z]+[y, z, x]+[z, x, y]+(x*y)*z+(y*z)*x+(z*x)*y=0$. 
\item[(LY4)]
$[x*y, z, w]+[y*z, x, w]+[z*x, y, w]=0$. 
\item[(LY5)]
$[x, y, z*w]=[x, y, z]*w+z*[x, y, w]$. 
\item[(LY6)]
$[x, y, [z, v, w]]=[[x, y, z], v, w]+[z, [x, y, v], w]+[z, v, [x, y, w]]$. 
\end{enumerate}
For $x, y\in T$, we define $D_{x, y}\in\mathrm{End}(T)$ by $D_{x, y}(z):=[x, y, z]$. 

(2)
A \emph{homomorphism of Lie-Yamaguti algebras} is 
a linear map $f: T\to T'$ 
satisfying $f(x*y)=f(x)*f(y)$ and $f([x, y, z])=[f(x), f(y), f(z)]$. 
\end{definition}

\begin{remark}
(1)
The notion of a Lie-Yamaguti algebra was defined by Yamaguti 
(\cite[Def. 2.1]{Yamaguti1958}), 
under the name ``general Lie triple system'', 
as the infinitesimal algebra of 
a reductive homogeneous space studied by Nomizu. 
Here we employ the definition from \cite{Yamaguti1969} 
which differs from \cite{Yamaguti1958} by a sign. 
The name ``Lie-Yamaguti algebra'' appeared in \cite{KW01}. 

(2)
In the rest of the paper, 
$k$ will be $\mathbb{R}$ or $\mathbb{C}$. 
In particular, it is of characteristic $0$ and 
therefore (LY1) is equivalent to $x*y=y*x$ for any $x, y\in T$ 
and (LY2) is equivalent to $[x, y, z]=-[y, x, z]$ for any $x, y, z\in T$. 
Note that $[\ ]$ works somewhat differently in the third argument. 

(3)
Using $D_{x, y}$, the conditions (LY2), (LY4), (LY5) and (LY6) 
can be expressed as follows, respectively: 
\begin{itemize}
\item[(LY2)']
$D_{x, x}=0$. 
\item[(LY4)']
$D_{x*y, z}+D_{y*z, x}+D_{z*x, y}=0$. 
\item[(LY5)']
$D_{x, y}(z*w)=D_{x, y}(z)*w+z*D_{x, y}(w)$, i.e. $D_{x, y}$ is a derivation for $*$. 
\item[(LY6)']
$D_{x, y}[z, v, w]=[D_{x, y}(z), v, w] + [z, D_{x, y}(v), w]+[z, v, D_{x, y}(w)]$, 
i.e. $D_{x, y}$ is a derivation for $[\ ]$.
\end{itemize}
\end{remark}

\begin{definition}
(\cite[Definition III.20]{Kowalski1980})
(1)
An \emph{infinitesimal $s$-manifold} 
is a pair $(T, \sigma)$ of a Lie-Yamaguti algebra $T$ 
and a linear map $\sigma: T\to T$ 
satisfying the following. 
\begin{itemize}
\item[(ISM0)]
Both $\sigma$ and $id_T-\sigma$ are invertible. 
\item[(ISM1)]
$\sigma(x * y) = \sigma(x) * \sigma(y)$. 
\item[(ISM2)]
$\sigma([x, y, z]) = [\sigma(x), \sigma(y), \sigma(z)]$. 
\item[(ISM3)]
$\sigma([x, y, z]) = [x, y, \sigma(z)]$. 
\end{itemize}

(2)
A \emph{homomorphism of infinitesimal $s$-manifolds} $(T, \sigma)$ and $(T', \sigma')$ is 
a homomorphism $f: T\to T'$ of Lie-Yamaguti algebras
satisfying $f\circ \sigma=\sigma'\circ f$. 
\end{definition}

\begin{remark}
(1)
The conditions 
(ISM1), (ISM2) and the first half of (ISM0) say that $\sigma$ is an automorphism of $T$. 
The condition (ISM3) (for all $z\in T$) can be expressed as $\sigma(D_{x, y})=D_{x, y}$, 
where $\sigma(f):=\sigma\circ f\circ\sigma^{-1}$ 
for $f\in\mathrm{End}(T)$. 

(2)
In \cite[Definition III.20]{Kowalski1980}, a more geometric term ``tensor'' is employed 
rather than ``operation.'' 
The notations correspond as follows: 
\begin{eqnarray*}
S_0(x) & = & \sigma(x), \\
\tilde{T}_0(x, y) & = & -x * y, \\
\tilde{R}_0(x, y)(z) & = & -[x, y, z]. 
\end{eqnarray*}

The conditions (LY1), (LY2) correspond to 
the condition (iv) in \cite[Proposition III.19]{Kowalski1980} ; 
(LY3) to (v); 
(LY4) to (vi); 
(LY5), (LY6) and (ISM3) to (ii); 
(ISM0) to (i); 
and (ISM1) and (ISM2) to (iii). 
\end{remark}

\begin{theorem}\label{thm_reg_s_mfd}
If $Q$ is a connected regular $s$-manifold and $q\in Q$, 
the following hold. 

(1) (\cite[Proposition II.33]{Kowalski1980})
As a quandle, $Q$ is transitive. 

(2) 
(\cite[Theorem II.32, Proposition II.38]{Kowalski1980})
The group $G=\mathrm{Tr}(Q)$ 
can be endowed with a structure of a connected Lie group 
with a smooth and transitive action on $Q$. 
There is a natural isomorphism between 
$(G/G_q, \rhd_{\mathrm{conj}_G(s_q)})$ and $Q$. 

(3)
(\cite[Theorem III.22]{Kowalski1980})
One can endow $T_qQ$ with natural operations $*$ and $[\ ]$ 
so that $(T_qQ, d_qs_q)$ is an infinitesimal $s$-manifold. 

(4)
(\cite[Proposition III.19]{Kowalski1980})
Regular $s$-manifolds $Q$ and $Q'$ are locally isomorphic 
near $q\in Q$ and $q'\in Q'$
if and only if $(T_qQ, d_qs_q)$ and 
 $(T_{q'}Q', d_{q'}s_{q'})$ are isomorphic. 

(5) (\cite[Theorem III.25]{Kowalski1980})
For any infinitesimal $s$-manifold $(T, \sigma)$, 
there exist a connected and simply-connected regular $s$-manifold $Q$ 
and an isomorphism $(T_qQ, d_qs_q)\cong (T, \sigma)$. 

It is unique up to a unique isomorphism: 
If $Q'$ also satisfies the condition with $q'\in Q'$, 
there is a unique isomorphism of $Q$ onto $Q'$ 
that sends $q$ to $q'$ 
and is compatible with the isomorphisms $T_qQ\cong T$ and $T_{q'}Q'\cong T$. 
\end{theorem}

\begin{remark}
(1)
For the uniqueness of the isomorphism in (5), 
see also \cite[Theorem III.22, Proposition III.23, Appendix B.7]{Kowalski1980}. 

(2)
Exactly analogous statements hold in the complex analytic case. 

In the algebraic case (over an algebraically closed field in characteristic $0$), 
the statements (1)--(3) hold. 
\end{remark}

In the next subsection, we will see that 
the assignment in (3) can be made into a functor: 
To a homomorphism (not only to an isomorphism as in (4)) of regular $s$-manifolds, 
we can associate a homomorphism of infinitesimal $s$-manifolds. 

For that, we will need some construction from the proof of (3)--(5) of this theorem. 
Let us start with some more definitions. 
\begin{definition}
(1) (\cite{Fedenko1977})
A \emph{$\Phi$-space} is a triplet $(G, H, \Phi)$ 
satisfying the following conditions: 
\begin{itemize}
\item
$G$ is a connected Lie group. 
\item
$\Phi$ is an endomorphism of $G$. 
\item
$H$ is a subgroup of $G^\Phi$ containing the connected component 
$G^\Phi_0$ of $G^\Phi$ at the identity. 
\end{itemize}

A $\Phi$-space $(G, H, \Phi)$ is called a \emph{regular $\Phi$-space} 
if $\Phi$ is an automorphism 
and $\mathrm{Lie}(H)$ (which is equal to $\mathrm{Lie}(G)^{d_e\Phi}$) 
is equal to the \emph{generalized} eigenspace of 
$d_e\Phi\in\mathrm{End}(\mathrm{Lie}(G))$ for $1$. 

If $\Phi$ is an automorphism, 
the quotient space $G/H$ with the operation $\rhd_\Phi$ 
defined in Proposition \ref{prop_conn_quandle} (1) 
is called a \emph{homogeneous $\Phi$-space}. 
If furthermore $(G, H, \Phi)$ is a regular $\Phi$-space, 
then $(G/H, \rhd_\Phi)$ is called 
a \emph{homogeneous regular $\Phi$-space}. 

(1)'
A \emph{weak $\Phi$-space} is a triplet $(G, H, \Phi)$ 
satisfying the following conditions: 
\begin{itemize}
\item
$G$ is a connected Lie group. 
\item
$\Phi$ is an endomorphism of $G$. 
\item
$H$ is a closed subgroup contained in $G^\Phi$. 
\end{itemize}
A \emph{transitive weak $\Phi$-space} 
is a weak $\Phi$-space where $\Phi$ is an automorphism and 
$(G/H, \rhd_\Phi)$ is a transitive quandle. 

If $\Phi$ is an automorphism, 
the $(G/H, \rhd_\Phi)$ 
is called a \emph{homogeneous weak $\Phi$-space}. 
It is called a \emph{homogeneous transitive weak $\Phi$-space}
if it is transitive as a quandle. 

(2)
(\cite{Fedenko1977})
A \emph{homomorphism} of (weak) $\Phi$-spaces 
$(G, H, \Phi)$ to $(G', H', \Phi')$ is a homomorphism of Lie groups 
$G\to G'$ which maps $H$ into $H'$ and commutes with $\Phi$ and $\Phi'$. 
\end{definition}

We make similar definitions in the complex analytic and algebraic categories.

\begin{definition}
(1)
For a smooth quandle (resp. a complex analytic quandle or a quandle variety) $Q$, 
a \emph{(weak) $\Phi$-space for $Q$} is 
a pair of a (weak) $\Phi$-space $(G, H, \Phi)$ with $\Phi$ an automorphism and 
an isomorphism $(G/H, \rhd_\Phi) \to (Q, \rhd)$. 

(2)
Let $(G, H, \Phi)$ and $(G', H', \Phi')$ 
be (weak) $\Phi$-spaces 
for $Q$ and $Q'$ 
given by isomorphisms $\alpha: G/H\to Q$ and $\alpha': G'/H'\to Q'$. 
For a quandle homomorphism $F: Q\to Q'$, 
a \emph{homomorphism $\tilde{F}: (G, H, \Phi)\to (G', H', \Phi')$ 
of (weak) $\Phi$-spaces over $F$} 
is a homomorphism of weak $\Phi$-spaces, given by $\tilde{F}: G\to G'$, 
which is compatible with $F$: Namely, 
if $\pi: G\to G/H$ and $\pi':  G'\to G'/H'$ 
are the natural projections, then 
$F\circ \alpha\circ\pi = \alpha'\circ\pi'\circ \tilde{F}$. 
\end{definition}

\begin{remark} \label{rem_phi_sp}
(1)
Note that, if $Q$ is a regular $s$-manifold 
and $(G, H, \Phi)$ is a weak $\Phi$-space for $Q$, 
then $1$ is not an eigenvalue of the endomorphism induced by 
$d_e\Phi$ on $\mathrm{Lie}(G)/\mathrm{Lie}(H)$, 
and thus $(G, H, \Phi)$ is in fact a regular $\Phi$-space. 

(2)
From the previous theorems, 
a connected transitive smooth quandle $Q$ 
(or complex analytic quandle or quandle variety in characteristic $0$) 
is a homogeneous weak $\Phi$-space. 

More precisely, if we take a point $q\in Q$ and let $G=\mathrm{Tr}(Q)$, 
then $(G, G_q, \mathrm{conj}_G(s_q))$ 
together with the map $G/G_q\to Q; gG_q\mapsto g\cdot q$ 
is a transitive weak $\Phi$-space for $Q$. 
If $Q$ is a connected regular $s$-manifold, 
then $(G, G_q, \mathrm{conj}_G(s_q))$ is a regular $\Phi$-space for $Q$. 
\end{remark}

\begin{definition}
(\cite{Fedenko1977}, \cite[Definition III.29]{Kowalski1980})
(1) 
A \emph{local regular $s$-triplet} is a triplet $(\fg, \fh, \varphi)$ 
satisfying the following conditions: 
\begin{itemize}
\item
$\fg$ is a Lie algebra. 
\item
$\varphi$ is an automorphism of $\fg$. 
\item
$\fh=\fg^\varphi$, 
and it is also equal to the generalized eigenspace of $\varphi$ for $1$. 
\end{itemize}
Note that we actually do not have to specify $\fh$. 

(2)
A \emph{homomorphism $(\fg, \fh, \varphi)\to (\fg', \fh', \varphi')$ of local regular $s$-triplets} 
is a homomorphism of Lie algebras 
$\fg\to \fg'$ which commutes with $\varphi$ and $\varphi'$ 
(and hence maps $\fh$ into $\fh'$). 
\end{definition}

We can associate an infinitesimal $s$-manifold 
to a local regular $s$-triplet, and vice versa, as follows. 
Restricting to ``prime'' local regular $s$-triplets, 
we have a one-to-one correspondence(\cite[Theorem III.30]{Kowalski1980}). 

\begin{construction}\label{const_ism_from_lrst}
Let $(\fg, \fh, \varphi)$ be a local regular $s$-triplet. 
There is a canonical decomposition $\fg=\fm\oplus\fh$ as a vector space, 
where $\fm$ is the sum of 
the generalized eigenspaces of $\varphi$ for eigenvalues different from $1$. 
We see that $[\fh, \fm]\subseteq \fm$ holds: 
In other words, $(\fg, \fm, \fh)$ is a reductive triple. 
For an element $X\in \fg$, 
let $X_\fm$ and $X_\fh$ denote 
its $\fm$-component and $\fh$-component, 
respectively. 
Then we have $\varphi(X_\fm)=\varphi(X)_\fm$ and $\varphi(X_\fh)=\varphi(X)_\fh=X_\fh$. 

We define operations on $\fm$, 
and hence on $\fg/\fh$ via the natural identification $\fm\cong \fg/\fh$, 
as follows: For $x, y, z\in\fm$, 
\begin{eqnarray*}
\sigma(x) & := & \varphi(x), \\
x*y & := & [x, y]_\fm, \\
 {[}x, y, z{]}  & := & [[x, y]_\fh, z]. 
\end{eqnarray*}
(Note that $\varphi(\fm)\subseteq\fm$ and 
$[\fh, \fm]\subseteq \fm$ hold.)
Then $\fm$ with the operations $*$ and $[\ ]$ is a Lie-Yamaguti algebra. 
It is straightforward to see that $\sigma$ is an automorphism on $(\fm, *, [\ ])$ 
and that $\sigma([x, y, z])=[x, y, \sigma(z)]$ holds. 
The eigenvalues of $\varphi$ on $\fm$ are all different from $1$, 
and it follows that $id_\fm-\sigma$ is invertible. 
Thus we have a structure of an infinitesimal $s$-manifold on $\fm$, or $\fg/\fh$. 
\end{construction}

\begin{definition}
(1)
For an infinitesimal $s$-manifold $(T, \sigma)$, 
a pair of a local regular $s$-triplet $(\fg, \fh, \varphi)$ and 
an isomorphism $\fg/\fh\to T$ 
is called a \emph{local regular $s$-triplet for $T$}. 

(2)
Given local regular $s$-triplets
$(\fg, \fh, \varphi)$ and $(\fg', \fh', \varphi')$  for $T$ and $T'$, respectively, 
and a homomorphism $f: T\to T'$ 
of infinitesimal $s$-manifolds, 
a \emph{homomorphism of local regular $s$-triplets 
$(\fg, \fh, \varphi)\to (\fg', \fh', \varphi')$ over $f$} 
is a  homomorphism of local regular $s$-triplets 
which induces $f$. 
\end{definition}

\begin{construction}\label{const_lrst_from_ism}
For a Lie-Yamaguti algebra $T$, 
a reductive triple inducing $T$ 
can be constructed as follows (\cite[pp. 61--62]{Nomizu1954}, \cite[Prop. 2.1]{Yamaguti1958}). 
We write $D_{x, y}\in\mathrm{End}(T)$ for the map 
$z\mapsto [x, y, z]$. 
\begin{itemize}
\item
$\mathfrak{h}(T)$ is the linear subspace of $\mathrm{End}(T)$ 
generated by $\{D_{x, y} \mid x, y\in T\}$, 
which is in fact a Lie subalgebra of $\mathrm{End}(T)$. 
\item
$\mathfrak{g}(T):=T\oplus \mathfrak{h}(T)$. 
\item
$[(x, u), (y, v)] := (x*y+u(y)-v(x), D_{x, y} +[u, v])$. 
\end{itemize}
It is straightforard to check that this defines a Lie algebra 
and that $(\fg(T), T, \fh(T))$ is a reductive triple, 
using the axioms of Lie-Yamaguti algebras. 
(A sign in the definition of $[\ ]$ is different from the one in \cite{Yamaguti1958} 
since we are following the conventions in \cite{Yamaguti1969}.)

An automorphism $\sigma$ of $T$ extends to an automorphism $\varphi_\sigma$ of $\fg(T)$ 
by $\varphi_\sigma(x, f)=(\sigma(x), \sigma\circ f\circ \sigma^{-1})$. 
If $(T, \sigma)$ is an infinitesimal $s$-manifold, 
then $\sigma([x, y, z])=[x, y, \sigma(z)]$ holds for any $x, y, z\in T$. 
This is equivalent to $\sigma\circ D_{x, y}\circ\sigma^{-1}=D_{x, y}$, 
hence $\varphi_\sigma|_{\fh(T)}=id_{\fh(T)}$ holds. 
Since $\sigma$ does not have $1$ as an eigenvalue, 
we see that $(\fg(T), \fh(T), \varphi_\sigma)$ is a local regular $s$-triplet. 
(This is a ``prime'' local regular $s$-triplet in the sense that 
$\fh(T)$ does not contain a proper ideal of $\fg(T)$ 
and $[T, T]\cap \fh(T)=\fh(T)$: 
See \cite[Theorem III.30]{Kowalski1980}.) 

It is easy to see that $(\fg(T), \fh(T), \varphi_\sigma)$ is a local regular $s$-triplet for $T$. 
\end{construction}

Now let $Q$ be a connected regular $s$-manifold. 
We take a regular $\Phi$-space $(G, H, \Phi)$ for $Q$, 
and let $q\in Q$ be the point corresponding to the residue class of $e\in G$. 
Then $(\mathrm{Lie}(G), \mathrm{Lie}(H), d_e\Phi)$ 
is a local regular $s$-triplet, 
and therefore we have a structure of 
an infinitesimal $s$-manifold 
on $\mathrm{Lie}(G)/\mathrm{Lie}(H)\cong T_qQ$. 
This is independent of the choice of a $\Phi$-space. 
(In fact, there is also a differential-geometric description 
of the infinitesimal algebra 
using a canonical affine connection on $Q$. 
See \cite[Chapter III]{Kowalski1980}.)

Conversely, if $(T, \sigma)$ is an infinitesimal $s$-manifold, 
let $(\fg(T), \fh(T), \varphi_\sigma)$ be the associated local regular $s$-triplet, 
$G$ the connected and simply-connected Lie group 
corresponding to $\fg(T)$ 
and $H$ the connected Lie subgroup of $G$ 
corresponding to $\fh(T)$. 
The automorphism $\varphi_\sigma$ of $\fg(T)$ 
defines an automorphism $\Phi$ of $G$, 
and since $H$ is a subgroup of $G^\Phi$ with the same Lie algebra, 
we see that $H$ is a closed subgroup. 
Thus $(G, H, \Phi)$ is a regular $\Phi$-space, 
and $(G/H, \rhd_\Phi)$ is a connected and simply-connected regular $s$-manifold 
whose infinitesimal algebra is isomorphic to $(T, \sigma)$.

\subsection{Homomorphisms of transitive quandle spaces}

Although some part of the following theorem is a combination of 
results in \cite{Fedenko1977}, 
we will give a detailed proof 
since the monograph \cite{Fedenko1977} is of limited availability.

\begin{theorem}[See \cite{Fedenko1977}]\label{thm_hom}
Let $Q$ and $Q'$ be connected regular $s$-manifolds 
and take $q\in Q$, $q'\in Q'$. 

(1)
If $F: Q\to Q'$ is a homomorphism with $F(q)=q'$, 
then $d_qF: T_qQ\to T_{q'}Q'$ is a homomorphism of 
infinitesimal $s$-manifolds. 

If $F_1, F_2: Q\to Q'$ are two homomorphisms with $F_1(q)=F_2(q)=q'$ and $d_qF_1=d_qF_2$, 
then $F_1=F_2$ holds. 

(2)
In the smooth or complex analytic case, 
if $f: T_qQ\to T_{q'}Q'$ is a homomorphism of 
infinitesimal $s$-manifolds, 
then there is a local homomorphism $F: (Q, q)\to (Q', q')$ 
such that $d_qF=f$. 

If $Q$ is simply-connected, one can take a global homomorphism. 

\end{theorem}

Combined with Theorem \ref{thm_reg_s_mfd}, 
this implies the following. 
\begin{corollary}\label{cor_functor}
In the smooth or complex analytic setting, 
let $\mathbf{Crsm}_*$ be the category of pointed connected regular $s$-manifolds, 
$\mathbf{SCrsm}_*$ the category of pointed simply-connected and connected regular $s$-manifolds, 
and $\mathbf{Ism}$ the category of infinitesimal $s$-manifolds. 

The assignment $(Q, q)\mapsto T_qQ$ gives an essentially surjective faithful functor 
$\mathcal{T}: \mathbf{Crsm}_*\to \mathbf{Ism}$,  
and its restriction $\mathbf{SCrsm}_*\to \mathbf{Ism}$ is an equivalence. 
\end{corollary}

\begin{remark}
In the algebraic case, we can say that the assignment gives a faithful functor. 
\end{remark}

For the proof, 
we first study 
homomorphisms of transitive quandle spaces 
in a little more generality than is needed for the proof 
of the theorem.

We begin with the case of a \emph{surjective} quandle homomorphism, 
although this is not logically required for the proof of the theorem. 
In this case, we have natural homomorphisms between 
the groups of inner automorphisms and transvections, respectively. 

\begin{proposition}
Let $F: Q\to Q'$ be a surjective homomorphism of quandles. 

(1)
There are naturally induced surjective homomorphisms 
$\mathrm{Inn}(F): \mathrm{Inn}(Q)\to \mathrm{Inn}(Q')$ and 
$\mathrm{Tr}(F): \mathrm{Tr}(Q)\to \mathrm{Tr}(Q')$, 
characterized by $\mathrm{Inn}(F)(s_q)=s_{F(q)}$ 
and $\mathrm{Tr}(F)(s_q^{-1}s_r)=s_{F(q)}^{-1}s_{F(r)}$ 
(and, more generally,  
$\mathrm{Tr}(F)(s_{q_1}^{e_1}\cdots s_{q_k}^{e_k})=
s_{F(q_1)}^{e_1}\cdots s_{F(q_k)}^{e_k}$ for $\sum e_i=0$). 

These homomorphisms are compatible with 
the actions of $\mathrm{Inn}(Q)$, $\mathrm{Inn}(Q')$, 
$\mathrm{Tr}(Q)$ and $\mathrm{Tr}(Q')$ on $Q$ and $Q'$. 

(2)
If furthermore $Q$ and $Q'$ are connected and transitive smooth quandles  
(resp. complex analytic quandles or quandle varieties) 
and $F$ is a $C^\infty$-map (resp. holomorphic or regular map), 
then $\mathrm{Inn}(F)$ and $\mathrm{Tr}(F)$ are 
$C^\infty$-maps (resp. holomorphic or regular maps). 

Consequently, if we choose $q\in Q$, 
let $G=\mathrm{Tr}(Q)$, $G'=\mathrm{Tr}(Q')$, 
$\Phi= \mathrm{conj}_G(s_q)$
and $\Phi'= \mathrm{conj}_{G'}(s_{F(q)})$
and think of 
$(G, G_q, \Phi)$ and $(G', G'_{F(q)}, \Phi')$ 
as transitive weak $\Phi$-spaces for $Q$ and $Q'$, 
then $\mathrm{Tr}(F)$ is a homomorphism 
of weak $\Phi$-spaces over $F$. 
\end{proposition}

\begin{proof}
(1)
The fact that $F$ induces 
such a homomorphism $\mathrm{Inn}(Q)\to \mathrm{Inn}(Q')$ 
is proven in \cite[Proposition 2.26]{Eisermann14}. 
This homomorphism maps $\mathrm{Tr}(Q)$ into $\mathrm{Tr}(Q')$. 
The surjectivity is obvious. 

The statement on the compatibility of group actions follows from 
$F(s_q(r))=s_{F(q)}(F(r))$, 
which is nothing but the condition that $F$ is a quandle homomorphism. 

(2)
We prove the first assertion for inner automorphism groups. 
The case of the groups of transvections is similar. 

From Sard's theorem, the transitivity of the action of $\mathrm{Inn}(Q)$ on $Q$ 
and the compatibility of the actions, 
we see that $F$ is submersive. 
For any $q'\in Q'$, 
let us take a neighborhood $U'$ of $q'$ 
and a local section $\sigma: U'\to Q$ of $F$. 
(In the algebraic case, take an \'etale neighborhood.) 
Then the action of $\mathrm{Inn}(Q)$ on $Q'$ is locally 
given by the composite map  
\[
\mathrm{Inn}(Q)\times U'\cong 
\mathrm{Inn}(Q)\times \sigma(U')\hookrightarrow \mathrm{Inn}(Q)\times Q\to Q\to Q', 
\]
hence is of class $C^\infty$ (resp. holomorphic or regular). 
This implies that $\mathrm{Inn}(F)$ is smooth, holomorphic or regular, respectively. 
(In the algebraic case, 
it is also possible to use Zariski's Main Theorem.) 

The assertion of the second paragraph is a consequence 
of the compatibility of the actions. 
\end{proof}

In the general case, 
we have a group homomorphism 
after extending the group acting on $Q$. 

\begin{proposition}\label{prop_group_hom_extended}
Let $Q$ and $Q'$ be quandles, 
$G=\mathrm{Inn}(Q)$ and $G'=\mathrm{Inn}(Q')$ 
(resp. $G=\mathrm{Tr}(Q)$ and $G'=\mathrm{Tr}(Q')$)
and let $F: Q\to Q'$ be a quandle homomorphism. 

(1)
Let $\ext{G}:=\{(g, g')\in G\times G' \mid F\circ g=g'\circ F\}\subseteq G\times G'$, 
and let $\pi_F: \ext{G}\to G$ and $\ext{F}: \ext{G}\to G'$ 
be the maps induced by the projection maps. 

Then $\ext{G}$ is a subgroup of $G\times G'$ 
which satisfies the following: 
\begin{itemize}
\item
$\pi_F: \ext{G}\to G$ is surjective. 
\item
$\ext{G}$ is invariant under $\mathrm{conj}_G(s_q)\times \mathrm{conj}_{G'}(s_{F(q)})$ 
for any $q\in Q$. 
\item
If we define 
an action of $\ext{G}$ on $Q$ via $\pi_F$ and 
an action of $\ext{G}$ on $Q'$ via $\ext{F}$, 
then $F$ is a $\ext{G}$-equivariant map. 
\item
For any $q\in Q$, 
the stablizer subgroup 
$\ext{G}_q$ (with respect to the above action) is 
elementwise fixed by $\mathrm{conj}_G(s_q)\times \mathrm{conj}_{G'}(s_{F(q)})$, 
and is mapped into $G'_{F(q)}$ 
by $\ext{F}: \ext{G}\to G'$. 
\end{itemize}

(2)
Assume furthermore that $Q$ and $Q'$ are transitive. 
Then, for any $q\in Q$, there is a natural identification 
of $(\ext{G}/\ext{G}_q, \rhd_{\ext{\Phi}})$ with $Q$, 
where $\ext{\Phi}$ is the restriction of 
$\mathrm{conj}_G(s_q)\times \mathrm{conj}_{G'}(s_{F(q)})$ to $\ext{G}$. 
The group homomorphism $\ext{F}$ commutes with $\ext{\Phi}$ and 
$\mathrm{conj}_{G'}(s_{F(q)})$, 
and it commutes with $F$. 

(3)
If $Q$ and $Q'$ are connected and transitive 
smooth quandles, complex analytic quandles or quandle varieties, 
then $\ext{G}$ is a closed subgroup of $G\times G'$. 

In the case $G=\mathrm{Tr}(Q)$ and $G'=\mathrm{Tr}(Q')$, 
let $\extc{G}$ be the connected component of $\ext{G}$ at the identity 
and $\extc{\Phi}$ the restriction of $\ext{\Phi}$ to $\extc{G}$. 
Then, we have a homomorphism 
$\extc{F}: (\extc{G}, \extc{G}_q, \extc{\Phi})\to (G', G'_{F(q)}, \mathrm{conj}_{G'}(s_{F(q)}))$ 
of weak $\Phi$-spaces over $F$. 
\end{proposition}

\begin{proof}
(1)
It is easy to see that 
$\ext{G}$ is a subgroup of $G\times G'$. 
From the assumption that $F$ is a quandle homomorphism, 
we have $F\circ s_x=s_{F(x)}\circ F$ 
(resp. $F\circ s_x\circ s_q^{-1}=s_{F(x)}\circ s_{F(q)}^{-1}\circ F$) 
for any $x\in Q$ (and $q\in Q$), 
and hence $(s_x, s_{F(x)})\in\ext{G}$ 
(resp. $(s_x\circ s_q^{-1}, s_{F(x)}\circ s_{F(q)}^{-1})\in\ext{G}$). 
Since $G$ is generated by $s_x$ (resp. $s_x\circ s_q^{-1}$), 
we see that the projection $\pi_F: \ext{G}\to G$ is surjective. 

The assertion that $F$ is $\ext{G}$-equivariant 
follows from the definition of $\ext{G}$. 

If we assume that $(g, g')\in \ext{G}$, i.e. $F\circ g=g'\circ F$, 
then we have 
\[
F\circ s_q\circ g\circ s_q^{-1} = s_{F(q)}\circ g'\circ s_{F(q)}^{-1}\circ F, 
\]
i.e. 
$(\mathrm{conj}_G(s_q)(g), \mathrm{conj}_{G'}(s_{F(q)})(g'))\in \ext{G}$. 
Thus $\ext{G}$ is invariant 
under $\mathrm{conj}_G(s_q)\times \mathrm{conj}_{G'}(s_{F(q)})$. 

Note that in general, 
for an automorphism $g$ of $Q$ satisfying $g(q)=q$, 
\[
s_q\circ g\circ s_q^{-1} = s_q\circ s_{g(q)}^{-1}\circ g = g 
\]
holds. 
Thus, if $\ext{g}=(g, g')\in\ext{G}$ is in the stabilizer $\ext{G}_q$ of $q$, 
then we have $\mathrm{conj}_G(s_q)(g)=g$. 
Since $F$ is $\ext{G}$-equivariant, 
$F(q)$ is fixed by $\ext{g}$, or $g'$. 
Hence $\mathrm{conj}_{G'}(s_{F(q)})(g')=g'$. 
Thus $\ext{g}$ is fixed by $\mathrm{conj}_G(s_q)\times \mathrm{conj}_{G'}(s_{F(q)})$ 
and $\ext{F}(\ext{g}) = g' \in G'_{F(q)}$. 

(2)
By (1), $\ext{\Phi}$ fixes $\ext{G}_q$ elementwise, 
and we can define a quandle $(\ext{G}/\ext{G}_q, \rhd_{\ext{\Phi}})$ 
as in Proposition \ref{prop_conn_quandle}. 

Since $\pi_F$ is surjective and $\pi_F^{-1}(G_q)=\ext{G}_q$, 
we have a bijective map $\ext{G}/\ext{G}_q \to G/G_q$. 
Since $\pi_F$ commutes with $\ext{\Phi}$ and $\mathrm{conj}_G(s_q)$, 
it gives a quandle isomorphism 
$(\ext{G}/\ext{G}_q, \rhd_{\ext{\Phi}})\cong 
(G/G_q, \rhd_{\mathrm{conj}_G(s_q)})\cong Q$. 

It is easy to see that $\ext{F}$ commutes with 
$\ext{\Phi}$ and $\mathrm{conj}_{G'}(s_{F(q)})$. 
To see the  compatiblity with $F$, 
let $\pi: \ext{G}\to \ext{G}/\ext{G}_q$ and $\pi': G'\to G'/G'_{F(q)}$ 
denote the the projections and 
$\alpha: \ext{G}/\ext{G}_q\cong Q$ 
and $\alpha': G'/G'_{F(q)}\cong Q'$ the isomorphisms determined 
by the actions of $\ext{G}$ and $G'$ on $Q$ and $Q'$, 
and $\ext{g}=(g, g')$ an element of $\ext{G}$, and then we have 
\[
F(\alpha(\pi(\ext{g})))
= F(g\cdot q)=g'\cdot F(q)=\ext{F}(\ext{g}) F(q)= \alpha'(\pi'(\ext{F}(\ext{g}))). 
\]

(3)
In the continuous case, 
for any $x\in Q$, the map $G\times G'\to Q'\times Q': (g, g')\mapsto (F(g(x)), g'(F(x)))$ 
is smooth, complex analytic or regular, respectively. 
Thus $F(g(x))=g'(F(x))$ defines a closed subset of $G\times G'$ 
(since the diagonal set in $Q'\times Q'$ is closed). 
Taking the intersection over all $x\in Q$, 
we see that $\ext{G}$ is a closed subgroup. 

In the case $G=\mathrm{Tr}(Q)$, 
it is a connected group and so 
the restriction of $\pi_F: \ext{G}\to G$ to the identity component $\extc{G}$ 
of $\ext{G}$ is surjective as well. 

It is easy to see that, when $\ext{G}$ is replaced by $\extc{G}$, 
other statements in (1) and (2) are also satisfied. 
Thus $(\extc{G}, \extc{G}_q, \extc{\Phi})$ is a weak $\Phi$-space for $Q$ 
and $\extc{F}:=\ext{F}|_{\extc{G}}$ 
is a homomorphism of weak $\Phi$-spaces over $F$. 
\end{proof}

\begin{remark}
We do not necessarily have a natural homomorphism $G\to G'$, 
even if $Q$ and $Q'$ are transitive. 
For inner automorphism groups, 
this has been pointed out in \cite[\S\S2.7]{Eisermann14}. 
Here we will give an example with 
$G=\mathrm{Tr}(Q)$ and $G'=\mathrm{Tr}(Q')$. 

Let $X_0=\begin{pmatrix} i & 0 \\ 0 & -i \end{pmatrix}$ 
and let $Q$ be the conjugacy class of $X_0$ 
in $\mathrm{SL}(2, \bC)$, 
considered as a quandle by the conjugation operation. 
Then let $Q'=Q\times \bC^2$ with 
$(X, a)\rhd (Y, b)=(XYX^{-1}, a + X(b-a))$. 

Then $Q$ and $Q'$ are connected regular $s$-manifolds, 
hence transitive smooth quandles. 

We embed $Q$ into $Q'$ as the zero section $Q\times\{0\}$. 
Take $X_1=-X_0\in Q$. 
Then $s_{X_1}\circ s_{X_0}^{-1}$ is the identity map on $Q$, 
while on $Q'$ we have 
$s_{(X_1, 0)}\circ s_{(X_0, 0)}^{-1}(Y, b)=(Y, -b)$. 
Thus the natural correspondence 
does not give a well-defined map. 

In this case, the transvection group of $Q$ is $\mathrm{PSL}(2, \bC)$, 
while the subgroup of $\mathrm{Tr}(Q')$ 
generated by transvections from $Q$ is $\mathrm{SL}(2, \bC)$. 
\end{remark}

Next, we look at the infinitesimal counterpart. 
We use the notations in Construction \ref{const_lrst_from_ism}. 
Recall in particular that $\fg(T)=T\oplus \fh(T)$ for a Lie-Yamaguti algebra $T$ 
and therefore that an element of $\fg(T)$ 
ican be written as $(x, u)$ with $x\in T$ and $u\in\mathrm{End}(T)$. 

\begin{proposition}\label{prop_alg_hom_extended}
(1) 
For a homomorphism $f: T\to T'$ of Lie-Yamaguti algebras, 
let 
\begin{eqnarray*}
\tilde{\fg} & := & \{((x, u), (x', u'))\in \fg(T)\oplus \fg(T')\mid f(x)=x', u'\circ f=f\circ u\}, \\
\tilde{\fm} & := & \{((x, 0), (x', 0))\in \fg(T)\oplus \fg(T')\mid f(x)=x'\}, \\
\tilde{\fh} & := & \{((0, u), (0, u'))\in \fg(T)\oplus \fg(T')\mid u'\circ f=f\circ u\}. 
\end{eqnarray*}
Then $\tilde{\fg}\subseteq \fg(T)\oplus \fg(T')$ is a Lie subalgebra, 
$(\tilde{\fg}, \tilde{\fm}, \tilde{\fh})$ is a reductive triple, 
and the projection maps $\pi_f: \tilde{\fg}\to \fg(T)$ 
and $\tilde{f}: \tilde{\fg}\to \fg(T')$ 
are homomorphisms of reductive triples. 
The map $\pi_f$ is surjective and induces an isomorphism $\tilde{\fm}\cong T$ 
of Lie-Yamaguti algebras, 
and $\tilde{f}$ induces $f$. 

If $f$ is surjective, then $\pi_f$ is an isomorphism 
and there is a Lie algebra homomorphism $\tilde{f}: \fg(T)\to\fg(T')$ 
which restricts to $f$. 

(2)
Let $f: (T, \sigma)\to (T', \sigma')$ 
be a homomorphism of infinitesimal $s$-manifolds. 
Then $\tilde{\fg}$ defined as above 
is invariant under $\varphi_\sigma\times\varphi_{\sigma'}$. 
If $\tilde{\varphi}$ denotes 
the restriction of $\varphi_\sigma\times\varphi_{\sigma'}$ to $\tilde{\fg}$, 
then $(\tilde{\fg}, \tilde{\fh}, \tilde{\varphi})$ 
is a local regular $s$-triplet for $(T, \sigma)$, 
and $\tilde{f}$ is a homomorphism of local regular $s$-triplets 
over $f$. 
\end{proposition}
\begin{proof}
(1)
It is clear that $\tilde{\fg}, \tilde{\fm}$ and $\tilde{\fh}$ are linear subspaces 
and that $\tilde{\fg}=\tilde{\fm}\oplus\tilde{\fh}$. 
So, we look at the following three types of commutators. 

For $((x, 0), (x', 0)), ((y, 0), (y', 0))\in\tilde{\fm}$, 
we have $f(x)=x'$ and $f(y)=y'$, 
and 
\[
[((x, 0), (x', 0)), ((y, 0), (y', 0))]
=((x*y, D_{x, y}), (x'*y', D_{x', y'})). 
\]
Since $f$ is a homomorphism, 
we have $f(x*y)=x'*y'$. 
Also, for any $z\in T$, 
\[
D_{f(x), f(y)}(f(z)) 
= [f(x), f(y), f(z)] = f([x, y, z]) = f(D_{x, y}(z)), 
\]
hence $D_{f(x), f(y)}\circ f = f\circ D_{x, y}$. 
Thus $((x*y, D_{x, y}), (x'*y', D_{x', y'}))$ belongs to $\tilde{\fg}$. 

If $((0, u), (0, u')), ((0, v), (0, v'))\in\tilde{\fh}$, 
then $u'\circ f=f\circ u$ and $v'\circ f=f\circ v$. 
We have 
\[
[((0, u), (0, u')), ((0, v), (0, v'))]
=((0, [u, v]), (0, [u', v'])), 
\]
and 
$[u', v']\circ f= u'\circ v'\circ f- v'\circ u'\circ f 
= f\circ u\circ v-f \circ v\circ u = f\circ [u, v]$. 
Thus $\tilde{\fh}$ is closed under the Lie bracket. 

If $((0, u), (0, u'))\in\tilde{\fh}$ and $((x, 0), (x', 0))\in\tilde{\fm}$, 
then $f(x)=x'$ and $u'\circ f=f\circ u$, 
and 
\[
[((0, u), (0, u')), ((x, 0), (x', 0))]
=((u(x), 0), (u'(x'), 0)), 
\]
which belongs to $\tilde{\fm}$ since $f(u(x))=u'(f(x))=u'(x')$. 
Thus $[\tilde{\fh}, \tilde{\fm}]\subseteq \tilde{\fm}$, and 
this concludes the proof of the statement that $\tilde{\fg}$ is a Lie subalgebra 
and 
$(\tilde{\fg}, \tilde{\fm}, \tilde{\fh})$ is a reductive triple. 

Obviously, the projections are Lie algebra homomorphisms, 
and map $\tilde{\fm}$ into $T$ and $T'$ 
and $\tilde{\fh}$ into $\fh(T)$ and $\fh(T')$. 
Hence they are homomorphisms of reductive triples. 

To show the surjectivity of $\pi_f$, 
it suffices to show that $(x, 0)$ and $(0, D_{x, y})$ 
are in the image for any $x, y\in T$. 
The former is clear from $((x, 0), (f(x), 0))\in\tilde{\fm}$. 
For the latter, 
we already saw that $D_{f(x), f(y)}\circ f = f\circ D_{x, y}$, 
which means that $((0, D_{x, y}), (0, D_{f(x), f(y)}))\in\tilde{\fh}$, 
and therefore that $(0, D_{x, y})\in\mathrm{Im}\ \pi_f$. 

It is clear that $\tilde{\fm}$ maps to $T$ bijectively. 
Since $\pi_f$ is a homomorphism of reductive triples, 
the induced map $\tilde{\fm}\to T$ is an isomorphism of Lie-Yamaguti algebras. 
Similarly, the map $\tilde{\fm}\to T'$ induced by $\tilde{f}$ 
is a homomorphism of Lie-Yamaguti algebras 
which can be identified with $f$. 

Finally, if $f$ is surjective, 
then the condition $u'\circ f=f\circ u$ determines $u'$ uniquely 
for a given $u$, 
and $\pi_f$ is an isomorphism.

(2)
The assumption means that $f\circ\sigma=\sigma'\circ f$. 
Then, for $((x, 0), (x', 0))\in\tilde{\fm}$ we have $f(x)=x'$ and 
\[
(\varphi_\sigma\times \varphi_{\sigma'})((x, 0), (x', 0)) 
= ((\sigma(x), 0), (\sigma'(x'), 0)), 
\]
and from $f(\sigma(x))=\sigma'(f(x))=\sigma'(x')$ 
we see that $\tilde{\fm}$ is stable under $\varphi_\sigma\times \varphi_{\sigma'}$.  

For $((0, u), (0, u'))\in\tilde{\fh}$, 
$u\in\fh(T)$ and $u'\in\fh(T')$ imply 
$\varphi_\sigma(u)=u$ 
and $\varphi_{\sigma'}(u')=u'$. 
Thus $((0, u), (0, u'))$ is fixed by $\varphi_\sigma\times \varphi_{\sigma'}$, 
and in particular $\tilde{\fh}$ is stable under $\varphi_\sigma\times \varphi_{\sigma'}$. 

Thus $\varphi_\sigma\times \varphi_{\sigma'}$ restricts to 
an automorphism $\tilde{\varphi}$ on $\tilde{\fg}$, 
and $\tilde{\fh}$ is elementwise fixed by $\varphi$. 
Since $\tilde{\varphi}|_{\tilde{\fm}}$ can be identified 
with $\sigma$ under the identification of $\tilde{\fm}$ with $T$, 
it does not have $1$ as an eigenvalue. 
Thus $(\tilde{\fg}, \tilde{\fh}, \tilde{\varphi})$ 
is a local regular $s$-triplet. 

Since $\pi_f$ commutes with $\tilde{\varphi}$ and $\varphi_\sigma$, 
it is a homomorphism of local regular $s$-triplets, 
and therefore $(\tilde{\fg}, \tilde{\fh}, \tilde{\varphi})$ 
is a local regular $s$-triplet for $(T, \sigma)$. 

Finally, $\tilde{f}$ commutes with $\tilde{\varphi}$ and $\varphi_{\sigma'}$, 
and induces $f$ on $\tilde{\fm}\cong T$, 
so it is a homomorphism of local regular $s$-triplets over $f$. 
\end{proof}

Now let us prove Theorem \ref{thm_hom}. 

\begin{proof}
(1)
Let 
\begin{itemize}
\item
$G=\mathrm{Tr}(Q)$, 
$\Phi=\mathrm{conj}_G(s_q)$, $\fg=\mathrm{Lie}(G)$, 
$\fh=\mathrm{Lie}(G_q)$, 
$\varphi=d_e\Phi$ and 
\item
$G'=\mathrm{Tr}(Q')$, 
$\Phi'=\mathrm{conj}_{G'}(s_{q'})$, 
$\fg'=\mathrm{Lie}(G')$, 
$\fh'=\mathrm{Lie}(G'_{q'})$, 
$\varphi'=d_e\Phi'$. 
\end{itemize}

By Proposition \ref{prop_group_hom_extended} (3) 
and Remark \ref{rem_phi_sp} (1), 
we have regular $\Phi$-spaces $(\tilde{G},\tilde{G}_q, \tilde{\Phi})$ 
and $(G', G'_{q'}, \Phi')$ for $Q$ and $Q'$ 
and a homomorphism $\tilde{F}: \tilde{G}\to G'$ 
of regular $\Phi$-spaces over $F$. 
Let $\tilde{\fg}=\mathrm{Lie}(\tilde{G})$, 
$\tilde{\fh}=\mathrm{Lie}(\tilde{G}_q)$ 
and $\tilde{\varphi}=d_e\tilde{\Phi}$. 
Then we have a homomorphism 
$d_e\tilde{F}: (\tilde{\fg}, \tilde{\fh}, \tilde{\varphi})\to (\fg', \fh', \varphi')$ 
of local regular $s$-triplets. 
This induces a homomorphism 
$\overline{d_e\tilde{F}}: \tilde{\fg}/\tilde{\fh} \to \fg'/\fh'$ of infinitesimal $s$-manifolds. 
The domain and the target can be identified with 
$T_qQ$ and $T_{q'}Q'$ as infinitesimal $s$-manifolds, 
and the map $\overline{d_e\tilde{F}}$ with $d_qF$. 
Thus $d_qF$ is a homomorphism of infinitesimal $s$-manifolds. 

Let $\fm$, $\tilde{\fm}$ and $\fm'$ denote 
the sum of eigenspaces of $d_e\Phi$, $d_e\tilde{\Phi}$ and $d_e\Phi'$ 
for the eigenvalues different from $1$. 
Then $d_e\pi_F$ restricts to an isomorphism $\alpha_F: \tilde{\fm}\to \fm$ 
since $d_e\pi_F$ commutes with $d_e\tilde{\Phi}$ and $d_e\Phi$. 
Similarly, 
$d_e\tilde{F}$ restricts to a map  $\beta_F: \tilde{\fm}\to\fm'$. 
Since the map $\fg/\fh \cong \tilde{\fg}/\tilde{\fh} \to \fg'/\fh'$ induced by $d_e\pi_F$ and $d_e\tilde{F}$ 
can be identified with $d_qF$, 
we may think of $\beta_F\circ \alpha_F^{-1}$ as $d_qF$. 
For a small neighborhood $U$ of $0$ in $\fm$, 
$\exp(v), \exp(\alpha_F^{-1}(v))$ and $\exp(d_qF(v))$ converge for any $v\in U$ 
and the map $U\to Q; v\mapsto\exp(v)q$ is an open embedding. 
Then, if $x = \exp(v)q$, 
\begin{eqnarray*}
F(x) & = & F(\exp(v)q)=F(\exp(\alpha_F^{-1}(v))q) \\
 & = & \tilde{F}(\exp(\alpha_F^{-1}(v)))\cdot F(q) =\exp(\beta_F(\alpha_F^{-1}(v)))q' \\
 & = & \exp((d_qF)(v))q'. 
\end{eqnarray*}
Thus $F(x)$ depends only on $(d_qF)(v)$. 

If $F_1, F_2: Q\to Q'$ are quandle homomorphisms, 
these arguments show that 
\[
\{x\in Q \mid F_1(x)=F_2(x), d_xF_1=d_xF_2\}
\]
is open. 
It is obviously closed. 
Thus, if $F_1(q)=F_2(q)$ and $d_qF_1=d_qF_2$, 
it is the whole $Q$, 
i.e. $F_1=F_2$. 

(2)
Since the universal cover of a regular $s$-manifold is 
again a regular $s$-manifold, 
it suffices to show the assertion for simply-connected $Q$ and $Q'$. 

Let $(T, \sigma)=(T_qQ, d_qs_q)$ and $(T', \sigma')=(T_{q'}Q', d_{q'}s_{q'})$. 
By Proposition \ref{prop_alg_hom_extended} (2), 
there is a local regular $s$-triplet $(\tilde{\fg}, \tilde{\fh}, \tilde{\varphi})$ 
for $(T, \sigma)$ 
and a homomorphism 
$\tilde{f}: (\tilde{\fg}, \tilde{\fh}, \tilde{\varphi}) \to (\fg(T'), \fh(T'), \varphi_{\sigma'})$ 
of local regular $s$-triplets over $f$. 

Let $\tilde{G}$ and $G'$ be the connected and simply-connected Lie groups 
corresponding to $\tilde{\fg}$ and $\fg(T')$, 
and $\tilde{\Phi}$ and $\Phi'$ their automorphisms corresponding to 
$\tilde{\varphi}$ and $\varphi_{\sigma'}$. 
Then $\tilde{G}^{\tilde{\Phi}}$ and $(G')^{\Phi'}$ are closed. 
Let $\tilde{H}$ and $H'$ be their connected components at the identity elements. 
By the uniqueness part of Theorem \ref{thm_reg_s_mfd} (5), 
$(\tilde{G}, \tilde{H}, \tilde{\Phi})$ and $(G', H', \Phi')$ are regular $\Phi$-spaces 
for $Q$ and $Q'$, respectively. 
More precisely, $\tilde{G}/\tilde{H}\to Q$ can be chosen so that 
the induced map $T_qQ=T\cong T_{e\tilde{H}}(\tilde{G}/\tilde{H})\to T_qQ$ 
is the identity map, and similarly for $Q'$. 
Let $\tilde{F}: \tilde{G}\to G'$ be the 
Lie group homomorphism corresponding to $\tilde{f}$. 
Then it is a homomorphism of regular $\Phi$-spaces, 
hence induces a homomorphism $F: Q\to Q'$ of regular $s$-manifolds, 
which satisfies $d_qF=f$. 
\end{proof}

\subsection{Ideals of a Lie-Yamaguti algebra and subalgebras of the associated Lie algebra}

We will later use the notion of an ideal of a Lie-Yamaguti algebra 
and some related facts. 
\begin{definition}(\cite[\S1]{Yamaguti1969})
Let $T$ be a Lie-Yamaguti algebra. 

(1) 
A \emph{subalgebra} of $T$ is a vector subspace $U\subseteq T$ 
satisfying $U*U\subseteq U$ and $[U, U, U]\subseteq U$. 

(2) 
An \emph{ideal} of $T$ is a vector subspace $U\subseteq T$ 
satisfying $U*T\subseteq U$ and $[U, T, T]\subseteq U$. 
(It follows that $T*U, [T, U, T], [T, T, U]$ are also contained in $U$.)

(3)
An ideal $U$ of $T$ is an \emph{abelian ideal} 
if $U*U=0$ and $[T, U, U]=0$. 
(It follows that $[U, T, U]=[U, U, T]=0$.)

(4)
If $V$ is another Lie-Yamaguti algebra, 
an \emph{extension} of $T$ by $V$ is an exact sequence 
$0\to V\to T^*\to T\to 0$ of vector spaces 
where $T^*$ is endowed with a structure of Lie-Yamaguti algebra 
and the mappings are homomorphisms. 
Note that the image of $V$ is an ideal of $T^*$. 
\end{definition}

The kernel of a homomorphism of Lie-Yamaguti algebras 
is an ideal, 
and the quotient space of a Lie-Yamaguti algebras 
by an ideal is a Lie-Yamaguti algebra in a natural way.

\begin{lemma}\label{lem_subalg_from_ideal}
For subspaces $U_1, U_2$ of $T$, 
let us denote by $\fh(U_1, U_2)$ 
the subspace of $\fh(T)$ generated by $\{D_{x, y}\mid x\in U_1, y\in U_2\}$. 

(1)
If $U$ is an ideal of $T$, 
then $U\oplus \fh(U, T)$ is 
an ideal of $\fg(T)=T\oplus \fh(T)$ 
and $U\oplus \fh(U, U)$ is 
an ideal of $U\oplus \fh(U, T)$. 
Also, $\fh(U, T)$ 
and $\fh(U, U)$ are ideals of $\fh(T)$, 
and hence are Lie subalgebras. 

(2)
If $f: T\to T'$ is a surjective homomorphism of Lie-Yamaguti algebras 
and $U=\mathrm{Ker}(f)$, 
then $U\oplus \fh(U, T)$ is contained in the kernel 
of the natural map $\tilde{f}: \fg(T)\to \fg(T')$. 

(3)
If $U$ is an abelian ideal of $T$, 
then $\fh(U, U)=0$ and 
the $U\oplus\{0\}$ is an abelian ideal of $U\oplus \fh(U, T)$. 
\end{lemma}
\begin{proof}
(1)
This is \cite[\S1, Proposition]{Yamaguti1969}. 
The proof is straightforward except that we use the identity 
$[D_{x, y}, D_{z, v}]=D_{[x, y, z], v}+D_{z, [x, y, v]}$, 
which is equivalent to (LY6). 

(2)
For $(x, u)\in T\oplus \fh(T)$, 
we have $\tilde{f}(x, u)=(f(x), u')$ with $u'\circ f=f\circ u$. 
Thus, if $x\in U$, $\tilde{f}(x, 0)=(0, 0)$ holds. 
We have only to show $u'=0$ 
for $u=D_{x, y}$ with $x\in U$ and $y\in T$. 
For $z\in T$, 
we have $u'(f(z))=f(u(z))=f([x, y, z])$, 
and this is $0$ since $[x, y, z]\in U=\mathrm{Ker}(f)$. 
From the surjectivity of $f$, 
we conclude that $u'=0$. 

(3)
For any $x, y\in U$, 
it follows from the definition of an abelian ideal 
that $D_{x, y}=0$, and hence $\fh(U, U)=0$. 
Thus $U$ is an ideal of $U\oplus \fh(U, T)$ by (1). 
For $x, y\in U$, 
we have $[(x, 0), (y, 0)]= (x*y, D_{x, y})=0$, 
and so $U$ is an abelian ideal of $U\oplus \fh(U, T)$. 
\end{proof}

\section{Quandle modules over quandle spaces}

As a candidate for the notion of modules over a rack or a quandle, 
one can consider an action of $Q$ on an additive group $M$, 
denoted by $x\rhd m$, 
and impose the condition $x\rhd(y\rhd m)=(x\rhd y)\rhd(x\rhd m)$. 
This is equivalent to considering modules over 
the associated group $\mathrm{As}(X)$, which is defined as the group 
generated by $\{g_x \mid x\in X\}$ subject to the condition 
$g_{x\rhd y}=g_xg_yg_x^{-1}$ 
(see \cite{EG2003}). 
A more general and natural notion of quandle modules and rack modules 
was introduced in \cite{AG2003} and \cite{Jackson2005}. 

\begin{definition}\label{def_modules}
Let $X$ be a rack. 
A rack module over $X$ is a triplet 
\[
((A_x)_{x\in X}, (\eta_{xy})_{x, y\in X}, (\tau_{xy})_{x, y\in X}), 
\]
where $A_x$ is an additive group for any $x\in X$, 
$\eta_{xy}: A_y\to A_{x\rhd y}$ is an isomorphism 
and 
$\tau_{xy}: A_x\to A_{x\rhd y}$ is a homomorphism 
for any $x, y\in X$, 
satisfying the following conditions: 
\begin{itemize}
\item[(1)]
$\eta_{x, y\rhd z}\eta_{yz}=\eta_{x\rhd y, x\rhd z}\eta_{xz}$, 
\item[(2)]
$\eta_{x, y\rhd z}\tau_{yz}=\tau_{x\rhd y, x\rhd z}\eta_{xy}$, 
\item[(3)]
$\tau_{x, y\rhd z}=\eta_{x\rhd y, x\rhd z}\tau_{xz}+\tau_{x\rhd y, x\rhd z}\tau_{xy}$. 
\end{itemize}
If $X$ is a quandle, 
it is called a quandle module if it further satisfies
$\eta_{xx}+\tau_{xx}=id_{A_x}$. 
\end{definition}

\begin{remark}
(1)
To see that this in fact generalizes 
the notion of a quandle module as an $\mathrm{As}(X)$-module, 
see Example \ref{ex_as_module}. 

(2)
Note that the terms in conditions (1), (2) and (3) 
exhaust naturally conceivable ways 
of going from $A_z$, $A_y$ and $A_x$ to $A_{x\rhd(y\rhd z)}=A_{(x\rhd y)\rhd(x\rhd z)}$ 
using $\eta$'s and $\tau$'s. 
\end{remark}

This definition is natural in that rack modules and quandle modules 
are in one-to-one correspondence with  
the abelian group objects in the categories of racks and quandles over $X$. 

\begin{definition}
Let $\mathbf{C}$ be a category with fiber products 
and $X$ an object of $\mathbf{C}$. 
An abelian group object in $\mathbf{C}$ over $X$ 
(also called a Beck module over $X$) is a tuple $(A, \pi, \zeta, \alpha, \iota)$ 
where
\begin{itemize}
\item
$A$ is an object of $\mathbf{C}$ and $\pi: A\to X$ is a momorphism, 
by which $A$ is regarded as an object over $X$, and 
\item
$\zeta: X\to A$, $\alpha: A\times_X A\to A$ 
and $\iota: A\to A$ are morphisms over $X$, 
\end{itemize}
satisfying the ``axiom of abelian groups'': 
\begin{itemize}
\item
$\alpha\circ(\alpha\times_X id_A)=\alpha\circ (id_A\times_X\alpha)$ 
as morphisms $A\times_X A\times_X A\to A$, 
\item
$\alpha\circ \sigma=\alpha$, where $\sigma: A\times_X A\to A\times_X A$ 
is the morphism interchanging the factors, 
\item
$\alpha\circ(id_A\times_X \zeta)=\alpha\circ(\zeta\times_X id_A)=id_A$ 
as endomorphisms of $A\cong A\times_X X\cong X\times_X A$, and 
\item
$\alpha\circ(id_A, \iota)=\alpha\circ(\iota, id_A)= \zeta\circ\pi$ 
as endomorphisms of $A$. 
\end{itemize}
A morphism $(A, \pi, \zeta, \alpha, \iota)\to (A', \pi', \zeta', \alpha', \iota')$ 
is defined as a morphism $f: A\to A'$ compatible with other data, 
i.e. $\pi'\circ f=\pi$, $f\circ \zeta=\zeta'$, 
$f\circ \alpha = \alpha'\circ(f\times_X f)$ and $f\circ \iota = \iota'\circ f$. 
\end{definition}

\begin{proposition}\label{prop_quandle_from_module}
(\cite[Theorem 2.18]{AG2003}, \cite[Theorem 2.2, Theorem 2.6]{Jackson2005}) 
Let $A=\coprod_{x\in X} A_x$, $\pi: A\to X$ the projection, 
$\zeta: X\to A$ the zero section, 
$\alpha: A\times_X A\to A$ and $\iota: A\to A$ the fiberwise operations 
of taking the sum and the inverse, respectively. 
We write $(x, a)$ to refer to an element $a\in A_x$ 
and define $(x, a)\rhd (y, b)=(x\rhd y, \eta_{xy}(b)+\tau_{xy}(a))$. 

Then $A$ is a rack, 
and $\pi$, $\zeta$, $\alpha$ and $\iota$ are rack homomorphisms. 
If $X$ is a quandle and $(A_x)_{x\in X}$ is a quandle module, 
then $A$ is a quandle and $\pi$, $\zeta$, $\alpha$ and $\iota$ are quandle homomorphisms. 

Furthermore, this correspondence gives an equivalence 
between the category of rack (or quandle) modules over $X$ and 
the category of abelian group objects in the category 
of racks (resp. quandles) over $X$. 
\end{proposition}

Over topological racks and quandles, 
the notion of a module was defined in \cite{EM2016}. 
We give a definition in a somewhat more abstract way, 
for topological quandles, quandle manifolds and quandle varieties, 
and their counterparts for racks. 

Let $(X, \mathcal{A})$ be one of the following: 
\begin{itemize}
\item
$X$ is a topological rack 
and $\mathcal{A}$ is a topological space 
together with a continuous map $\mathcal{A}\to X$ 
and an addition operation $\mathcal{A}\times_X\mathcal{A}\to X$ 
which makes $\mathcal{A}$ a relative topological additive group. 
\item
$X$ is a rack manifold (of a certain class) 
and $\mathcal{A}$ is a vector bundle over $X$. 
(This is a special case of the above.)
\item
$X$ is a rack variety 
and $\mathcal{A}$ is a variety over $X$ 
endowed with an addition operation $\mathcal{A}\times_X\mathcal{A}\to X$ 
which makes $\mathcal{A}$ a relative commutative group variety. 
\item
$X$ is a rack variety and $\mathcal{A}$ is a coherent sheaf on $X$. 
\end{itemize}

Let $\pi_i: X^2\to X$ and $p_i: X^3\to X$ 
denote the $i$-th projection maps, 
$p_{ij}=(p_i, p_j): X^3\to X^2$, 
$\mu: X^2\to X$ the quandle operation and 
$\mu_{ij}:=\mu\circ p_{ij}: X^3\to X$. 

\begin{definition}\label{def_modules_sp}
(1)
A \emph{rack module} over $X$ is a triplet 
$(\mathcal{A}, \eta, \tau)$, 
where $\mathcal{A}$ is as above, 
$\eta: \pi_2^*\mathcal{A}\to\mu^*\mathcal{A}$ is an isomorphism (in the appropriate category) 
over $X\times X$ 
and 
$\tau: \pi_1^*\mathcal{A}\to\mu^*\mathcal{A}$ is a homomorphism over $X\times X$, 
satisfying the following conditions: 
\begin{itemize}
\item[(a)]
$(p_1, \mu_{23})^*\eta\circ p_{23}^*\eta
=
(\mu_{12}, \mu_{13})^*\eta\circ p_{13}^*\eta$, 
\item[(b)]
$(p_1, \mu_{23})^*\eta\circ p_{23}^*\tau
=
(\mu_{12}, \mu_{13})^*\tau \circ p_{12}^*\eta$, 
\item[(c)]
$(p_1, \mu_{23})^*\tau
=
(\mu_{12}, \mu_{13})^*\eta\circ p_{13}^*\tau
+(\mu_{12}, \mu_{13})^*\tau\circ p_{12}^*\tau$. 
\end{itemize}
Here in (1), for example, 
the domains of $p_{23}^*\eta$ and 
$p_{13}^*\eta$ are 
$p_{23}^*\pi_2^*\mathcal{A}$ and $p_{13}^*\pi_2^*\mathcal{A}$ respectively, 
and both can be identified with (or are equal to, depending on how the pullbacks are defined) 
$p_3^*\mathcal{A}$. 
Similarly, the codomain of $p_{23}^*\eta$ and the domain of $(p_1, \mu_{23})^*\eta$ are 
$p_{23}^*\mu^* \mathcal{A}$ 
and $(p_1, \mu_{23})^*\pi_2^*\mathcal{A}$, 
both of which can be identified with 
$\mu_{23}^*\mathcal{A}$, 
so the two morphisms can be composed.

(2)
If $X$ is a quandle, 
a rack module $(\mathcal{A}, \eta, \tau)$ is called a \emph{quandle module} over $X$ 
if it further satisfies $\Delta^*\eta+\Delta^*\tau=id_{\mathcal{A}}$, 
where $\Delta: X\to X\times X$ is the diagonal map. 
Here we use the natural identifications 
$\Delta^*\pi_1^*\mathcal{A}\cong 
\Delta^*\pi_2^*\mathcal{A}\cong
\Delta^*\mu^*\mathcal{A}\cong \mathcal{A}$. 

\smallbreak
In the case $X$ is a manifold and $\mathcal{A}$ is a vector bundle
or $X$ is an algebraic variety and $\mathcal{A}$ is a coherent sheaf, 
we will call $(\mathcal{A}, \eta, \tau)$ a \emph{linear} rack or quandle module. 
\end{definition}

\begin{remark}
If $X$ is a transitive rack variety, 
a linear (coherent) rack module $\mathcal{A}$ over $X$ is always locally free of constant rank, 
and we will identify it with the associated vector bundle. 
\end{remark}

\begin{definition}\label{def_regular_modules}
Let $Q$ be a connected regular $s$-manifold 
in the smooth, complex analytic or algebraic category 
and $(\mathcal{A}, \eta, \tau)$ a linear quandle module over $Q$. 
We say $\mathcal{A}$ is \emph{regular} 
if $id_{\mathcal{A}_q} - \eta_{qq}=\tau_{qq}: \mathcal{A}_q\to \mathcal{A}_q$ 
is invertible for any (or, equivalently, some) $q\in Q$. 
\end{definition}

\begin{definition}
Let $\mathcal{A}$ and $\mathcal{B}$ 
be rack (resp. quandle) modules over $X$. 

(1)
A rack (resp. quandle) module homomorphism $f: \mathcal{A}\to \mathcal{B}$ 
is a continuous additive map, 
a bundle map of locally constant rank, 
a homomorphism of $\mathcal{O}_X$-modules, etc., 
compatible with $\eta$ and $\tau$. 

More specifically, if $\eta^\mathcal{A}$, $\tau^\mathcal{A}$, 
$\eta^\mathcal{B}$ and $\tau^\mathcal{B}$ denote 
the homomorphisms defining the module structures, 
we require that 
$\eta^\mathcal{B}\circ \pi_2^*(f) =\mu^*(f)\circ\eta^\mathcal{A}$ 
and 
$\tau^\mathcal{B}\circ \pi_1^*(f) =\mu^*(f)\circ\tau^\mathcal{A}$. 

(2)
A rack (resp. quandle) submodule $\mathcal{A}'$ of $\mathcal{A}$ 
is a subspace whose fibers are subgroups, 
a subbundle, a coherent submodule, etc., 
which is closed under the maps $\eta$ and $\tau$. 
Formally, 
the requirement is that 
$\eta$ and $\tau$ maps $\pi_2^*\mathcal{A}'$ and $\pi_1^*\mathcal{A}'$
to $\mu^*\mathcal{A}'$, respectively. 
(In the coherent case, note that $\mu=\pi_2\circ (\pi_1, \mu)$ is flat 
since $(\pi_1, \mu)$ is an isomorphism, 
and hence that $\pi_1^*\mathcal{A}'$, $\pi_2^*\mathcal{A}'$ and $\mu^*\mathcal{A}'$ 
can be considered as subsheaves of 
$\pi_1^*\mathcal{A}$, $\pi_2^*\mathcal{A}$ and $\mu^*\mathcal{A}$, 
respectively.) 
\end{definition}

\begin{remark}
If $X$ is a transitive, 
a bundle map compatible with $\eta$ is automatically of constant rank. 
\end{remark}

\begin{proposition}\label{prop_module_category}
Let $X$ be a rack or quandle space as above. 

(1)
Rack (resp. quandle) modules over $X$ 
form a category by the homomorphisms defined above. 

(2)
In the case $X$ is a manifold and modules are vector bundles, 
if $f$ is a rack (resp. quandle) module homomorphism of modules over $X$ 
which is a bundle map of locally constant rank, 
then $\mathrm{Ker} f$, $\mathrm{Im} f$, and $\mathrm{Coker} f$ 
are rack (resp. quandle) modules in a natural way. 

In particular, if $X$ is transitive, 
then linear rack (resp. quandle) modules form an abelian category. 

(3)
Coherent rack (resp. quandle) modules over a rack (resp. quandle) variety 
form an abelian category. 
\end{proposition}
\begin{proof}
(1), (2) Straightforward. 

(3)
We can show that 
the sheaf-theoretic kernel, image and cokernel of a rack homomorphism are 
again rack modules in a natural way, 
using the fact that $\mu$ (as well as $\pi_1$ and $\pi_2$) is flat as remarked above. 

Take $\mathcal{K}=\mathrm{Ker} f$, for example. 
Then, from the exact sequence $0\to \mathcal{K}\to \mathcal{A}\overset{f}\to\mathcal{B}$, 
we obtain the following commutative diagram with exact rows. 
\[
\xymatrix{
0 \ar[r] & \pi_2^*\mathcal{K} \ar[r] &  
 \pi_2^*\mathcal{A} \ar[r]^{\pi_2^*f} \ar[d]^{\eta^\mathcal{A}}_\cong &  
 \pi_2^*\mathcal{B} \ar[d]^{\eta^\mathcal{B}}_\cong \\
0 \ar[r] & \mu^*\mathcal{K} \ar[r] & 
 \mu^*\mathcal{A} \ar[r]^{\mu^*f}  &  \mu^*\mathcal{B}  \\
}
\]
Thus there is a unique isomorphism 
$\eta^\mathcal{K}:  \pi_2^*\mathcal{K}\to  \mu^*\mathcal{K}$ 
making the whole diagram commutative. 
We can find $\tau^\mathcal{K}$ in a similar way, 
and the conditions for $\mathcal{K}$ to be a rack module 
follow from those for $\mathcal{A}$. 
\end{proof}

\begin{remark}
Unlike the case of group representations, 
we do not seem to have tensor products and duals. 
\end{remark}

We state the correspondence of modules and abelian group objects 
in the case of linear modules over quandle spaces. 

\begin{definition}
Let $Q$ be a smooth quandle, complex analytic quandle 
or a quandle variety, 
and let $\mathbb{K}$ denote the base field. 

(1)
The category of linear quandle modules over $Q$ 
is denoted by $\mathbf{Mod}_{\mathbb{K}}(Q)$. 

(2)
We define the category $\mathbf{Vec}(\mathbf{Quandle}/Q)$ of vector bundles in the category 
of quandle spaces (in each setting) over $Q$ as follows. 
\begin{itemize}
\item
An object is a pair $(\mathcal{A}, \rhd_{\mathcal{A}})$, 
where 
\begin{itemize}
\item
$\mathcal{A}$ is endowed with a structure of a vector bundle over $Q$, 
\item
$\rhd_{\mathcal{A}}$ is a binary operation on $\mathcal{A}$ 
which makes $\mathcal{A}$ a smooth quandle, complex analytic quandle or quandle variety, 
\end{itemize} 
such that the projection map $\mathcal{A}\to Q$, 
the zero section $Q\to\mathcal{A}$, 
the fiberwise addition $\mathcal{A}\times_Q\mathcal{A}\to\mathcal{A}$, 
the fiberwise inversion $\mathcal{A}\to \mathcal{A}$ 
and the scalar multiplications $\mathcal{A}\to\mathcal{A}$ are 
homomorphism of quandle spaces. 
\item
A morphism is a linear bundle map which is also a quandle homomorphism. 
\end{itemize}
\end{definition}

\begin{remark}
The definition of the latter category is not entirely category-theoretic. 

We can easily see that $\mathcal{A}$ as in (2) is an abelian group object 
in the category of quandle spaces over $Q$, 
and furthermore is a $\mathbb{K}$-vector space object. 
However, in order to obtain an equivalence as in the following proposition, 
there has to be some kind of continuity condition on the scalar 
(especially when we extend to the case 
where the base field is of positive characteristic). 

A category theoretic formulation using a ``field object'' is possible, 
but we will skip it for brevity. 
\end{remark}

\begin{proposition}\label{prop_quandle_space_from_module}
Let $Q$ be transitive smooth quandle, 
complex analytic quandle or quandle variety 
and let $\mathbb{K}$ denote the base field. 

Then there is an equivalence of the categories 
$\mathbf{Mod}_{\mathbb{K}}(Q)$ 
and $\mathbf{Vec}(\mathbf{Quandle}/Q)$. 

An analogous statement holds for racks. 
\end{proposition}

\begin{proof}
In the algebraic case, 
it follows from the transitivity of $Q$ that 
a quandle module $\mathcal{A}$ over $Q$ is a locally free sheaf. 
Therefore, we can think of $\mathcal{A}$ as a vector bundle in any of the settings. 
Also note that homomorphisms are always of constant rank. 

For $x, y\in Q$, 
we write $\mathcal{A}_x$ for the fiber of $\mathcal{A}$ at $x$ 
and $\eta_{xy}$ for the map $\mathcal{A}_y\to\mathcal{A}_{x\rhd y}$ 
induced by $\eta$. 
We also write $(x, a)$ for an element $a\in\mathcal{A}_x$. 

Given an object $(\mathcal{A}, \eta, \tau)$ of $\mathbf{Mod}_{\mathbb{K}}(Q)$, 
the correspondence of Proposition \ref{prop_quandle_from_module} 
gives a quandle structure 
by $(x, a)\rhd_{\mathcal{A}} (y, b)=(x\rhd y, \eta_{xy}(b)+\tau_{xy}(a))$. 
This obviously gives a smooth, complex analytic or regular map 
$\mathcal{A}\times \mathcal{A}$. 
The map $(\rhd_{\mathcal{A}}, \pi_2): \mathcal{A}\times \mathcal{A}\to \mathcal{A}\times \mathcal{A}$ 
is an isomorphism, since its inverse is given by 
$((x, a), (y, b))\mapsto (x\rhd^{-1}y, \eta_{x, x\rhd^{-1}y}(b-\tau_{x, x\rhd^{-1}y}(a))$. 
Thus $(\mathcal{A}, \rhd_{\mathcal{A}})$ is a smooth quandle, etc. 

The projection, zero section, addition, and inversions are quandle homomorphisms 
by Proposition \ref{prop_quandle_from_module}, and the scalar multiplications 
are quandle homomorphisms from the above expression for $\rhd_{\mathcal{A}}$ 
and linearity of $\eta$ and $\tau$. 
They are smooth maps etc., so we have an object of 
$\mathbf{Vec}(\mathbf{Quandle}/Q)$. 

In both categories, a morphism is a (linear) bundle map. 
A bundle map $f$ is is a morphism in $\mathbf{Mod}_{\mathbb{K}}(Q)$ 
if and only if it is a homomorphism of modules over $Q$ as a discrete quandle. 
By Proposition \ref{prop_quandle_from_module}, 
this condition is equivalent to the condition that $f$ is a (discrete) quandle homomorphism 
compatible with the projection to $Q$, addition, etc., 
Thus we have a fully faithful functor 
$F: \mathbf{Mod}_{\mathbb{K}}(Q)\to\mathbf{Vec}(\mathbf{Quandle}/Q)$. 

Given an object $(\mathcal{A}, \rhd_{\mathcal{A}})$, 
by Proposition \ref{prop_quandle_from_module} 
we may write 
\[
(x, a)\rhd_{\mathcal{A}} (y, b)=(x\rhd y, \eta_{xy}(b)+\tau_{xy}(a))
\]
where $\eta_{xy}$ is an invertible additive map 
and $\tau_{xy}$ is an additive maps for $x, y\in Q$. 
From the smoothness etc. of $\rhd_{\mathcal{A}}$ 
it follows that $\eta_{xy}$ and $\tau_{xy}$ form bundle maps 
$\eta$ and $\tau$ on $Q\times Q$. 
They are $\mathbb{K}$-linear from the condition that 
scalar multiplications are homomorphisms for $\rhd_{\mathcal{A}}$. 
Thus we have a quandle module over $Q$. 
It is easy to see that this module is sent to $(\mathcal{A}, \rhd_{\mathcal{A}})$, 
and so the functor $F$ is an equivalence. 

\end{proof}

\begin{example}\label{ex_as_module}
Let $Q$ be a topological quandle, smooth quandle, complex analytic quandle or a quandle variety, 
$G$ a topological, Lie, complex Lie or algebraic group, 
$V$ a representation of $G$ 
and $\varphi: Q\to\mathrm{Conj}(G)$ a quandle homomorphism. 
(Here our convention for the conjugation quandle is 
such that $g\rhd h=ghg^{-1}$.) 
We take $\mathcal{A}=Q\times V$, 
$\eta_{xy}=\varphi(x)$ and $\tau_{xy}=1-\varphi(x\rhd y)$. 
Then $(\mathcal{A}, \eta, \tau)$ is a quandle module over $Q$. 
In the discrete setting, we may take $G$ to be 
the associated group $\mathrm{As}(Q)$. 
(See the paragraph before Def. 2.23 of [AG].)

In particular, if $G=\langle g\rangle$ is 
a cyclic group of infinite order, 
$f(x)=g$ for any $x\in Q$ 
and $V$ is a vector space (or an additive group) with a $G$-action, 
then we obtain a quandle module 
by $\eta_{xy}(a)=ga$ and $\tau_{xy}(a)=(1-g)a$ 
([AG, Example 2.20]). 

The quandle module structure above 
is the one given by the natural correspondences 
\begin{eqnarray*}
& & \{\hbox{representations of $G$}\} \\
& \leftrightarrow & 
\{\hbox{extensions of $G$ which are ``vector-spaces over $G$''}\} \\
& \to & 
\{\hbox{extensions of $Q$ which are ``vector-spaces over $Q$''}\} \\ 
& \leftrightarrow & 
\{\hbox{modules over $Q$}\}. 
\end{eqnarray*}
The first one is given by 
associating to a representation $V$ 
the group $G\times V$ with the multiplication
\[
(g, a)(h, b)=(gh, a+gb). 
\]
Then the conjugation is given by 
\[
(g, a)\rhd (h, b)=(g, a)(h, b)(g^{-1}, -g^{-1}a)
=(g\rhd h, gb+(1-g\rhd h)a). 
\]
The pullback of this quandle module over $G$ by $\varphi$ 
is exactly the quandle module $(\mathcal{A}, \eta, \tau)$ 
under the correspondence of Proposition \ref{prop_quandle_from_module}. 
\end{example}

\begin{example}
In the situation of the previous example, 
there is another natural module structure 
given by $\eta'_{xy}=\varphi(x)$ and $\tau'_{xy}=1-\varphi(x)$. 
In this case, the corresponding quandle structure on $\mathcal{A}$ 
is given by 
$(x, a)\rhd (y, b)=(x\rhd y, a+\varphi(x)(b-a))$. 

There is a homomorphism $f: (Q\times V, \eta', \tau') \to (Q\times V, \eta, \tau)$ 
given by $(x, a)\mapsto (x, (1-\varphi(x))(a))$, 
and the two modules are isomorphic 
if $1-\varphi(x)$ is invertible for any $x\in Q$. 
If not, 
the module structures are not necessarily isomorphic, 
as we will now see. 

For $\alpha\not=0, 1$, 
let 
$Q=\{X\in\mathrm{GL}(2, \mathbb{C})\mid \det X=\alpha, \mathrm{Tr}(X)=1+\alpha\}$, 
i.e. the conjugacy class of $X_0:=\begin{pmatrix}1 & 0 \\ 0 & \alpha \end{pmatrix}$. 
It is a regular $s$-manifold by the operation $X\rhd Y=XYX^{-1}$. 
(We remark that $Q$ is isomorphic to 
$\{X\in\mathrm{SL}(2, \mathbb{C})\mid \mathrm{Tr}(X)=\sqrt{\alpha}+1/\sqrt{\alpha}\}$ 
with the conjugation operation.) 
The standard representation of 
$\mathrm{SL}(2, \mathbb{C})$ gives rise to two quandle module structures 
$(\eta, \tau)$ and $(\eta', \tau')$ 
on $\mathcal{A}:=Q\times \mathbb{C}^2$. 
We write the associated quandle operations on $\mathcal{A}$ by 
$\rhd$ and $\rhd'$. 
It is easy to show that $(\mathcal{A}, \rhd)$ can be identified with the 
inverse image of $Q$ by 
$\mathrm{Conj}(\mathrm{Aff}(2, \mathbb{C}))\to \mathrm{Conj}(\mathrm{GL}(2, \mathbb{C}))$. 

We show that the two module structures are not isomorphic. 
First note that there is a section $Q\to \mathcal{A}$, 
given by $s(X)=(\alpha I-X)\begin{pmatrix} 1 \\ 0\end{pmatrix}$, 
which satisfies $(X, s(X))\rhd' (Y, 0)=(X\rhd Y, 0)$ for any $X$ and $Y$. 
If there were an isomorphism $(\mathcal{A}, \rhd')\overset{\sim}{\to} (\mathcal{A}, \rhd)$, 
given by $f_X\in \mathrm{GL}(2, \mathbb{C})$ 
for each $X\in Q$, 
then $t(X):=f_X(s(X))$ would satisfy 
\[
(X, t(X))\rhd (Y, 0)=(X\rhd Y, 0), 
\hbox{ i.e. } (I-XYX^{-1})t(X)=0, 
\]
for any $X, Y\in Q$. 
Since $s(X_0)\not=0$, we would have $a:=t(X_0)\not=0$. 
On the other hand, applying the above condition with $X=Y=X_0$, 
we see that $X_0a=a$. 
Using this, the above condition with $X=X_0$ implies $Ya=a$ for any $Y\in Q$, 
but this is a contradiction since 
the eigenspace of $Y$ for the eigenvalue $1$ 
can be any $1$-dimensional subspace of $\mathbb{C}^2$ 
as $Y$ moves. 
\end{example}

\begin{example}
Here we will give two examples of quandle modules over a quandle variety 
which are not locally free. 

Let $k$ be a field and 
let $Q=\{X\in\mathrm{SL}(2, k)\mid \mathrm{Tr}(Q)=2\}$ 
be the quandle variety given by the operation $X\rhd Y=XYX^{-1}$. 
We endow $\mathcal{A}:=Q\times k^2$ with the quandle module structure 
associated to the standard representation of $\mathrm{SL}(2, k)$ 
as in Example \ref{ex_as_module}. 
Denote by $\mathfrak{m}$ the maximal ideal of the affine coordinate ring of $Q$ 
corresponding to the identity matrix $I$. 
Then, for each positive integer $n$, $\mathfrak{m}^n\mathcal{A}$ 
is a quandle submodule of $\mathcal{A}$. 
By Proposition \ref{prop_module_category} (3), 
we have quandle modules $\mathcal{A}/\mathfrak{m}^n\mathcal{A}$ 
supported on $\{I\}$. 

For example, $\mathcal{A}/\mathfrak{m}\mathcal{A}$ 
is a skyscraper sheaf at $I$ with fiber $k^2$. 
The map $\eta_{XY}$ is $X$ if $Y=I$ and $0$ otherwise, 
and $\tau_{XY}$ is always $0$. 
\end{example}

\section{Representations and extensions of Lie-Yamaguti algebras 
and infinitesimal $s$-manifolds}

The following notion of a representation 
of a Lie-Yamaguti algebra was defined 
in \cite[\S6]{Yamaguti1969}. 

\begin{definition}\label{def_rly}
A \emph{representation of a Lie-Yamaguti algebra} $T$ is 
a quadruplet $(V, \rho, \delta, \theta)$, 
where $V$ is a vector space, $\rho: T\to\mathrm{End}(V)$ 
is a linear map and $D, \theta: T\times T\to\mathrm{End}(V)$ 
are bilinear maps, 
such that the following hold for any $x, y, z, w\in T$. 
\begin{enumerate}
\item[(RLY1)]
$\delta(x, y)+\theta(x, y)-\theta(y, x)=[\rho(x), \rho(y)]-\rho(x*y)$. 
\item[(RLY2)]
$\theta(x, y*z)-\rho(y)\theta(x, z)+\rho(z)\theta(x, y)=0$. 
\item[(RLY3)]
$\theta(x*y, z)-\theta(x, z)\rho(y)+\theta(y, z)\rho(x)=0$. 
\item[(RLY4)]
$\theta(z, w)\theta(x, y)-\theta(y, w)\theta(x, z)-\theta(x, [y, z, w])
+ \delta(y, z)\theta(x, w)=0$. 
\item[(RLY5)]
$[\delta(x, y), \rho(z)]=\rho([x, y, z])$, 
\item[(RLY6)]
$[\delta(x, y), \theta(z, w)] = \theta([x, y, z], w)+\theta(z, [x, y, w])$. 
\end{enumerate}
A \emph{homomorphism} of representations of Lie-Yamaguti algebras 
is a linear map compatible with $\rho, \delta$ and $\theta$. 
\end{definition}

\begin{remark}
(1)
As pointed out in \cite{Yamaguti1969}, 
we may do without $\delta$ if we take (RLY1) as the definition of $\delta$, 
and that the following holds: 
\begin{enumerate}
\item[(RLY7)]
$\delta(x*y, z)+\delta(y*z, x)+\delta(z*x, y) = 0$. 
\end{enumerate}

(2)
We also have 
\begin{enumerate}
\item[(RLY8)]
$[\delta(x, y), \delta(z, w)] = \delta([x, y, z], w)+\delta(z, [x, y, w])$. 
\end{enumerate}
To show this, we use (RLY1) to write $\delta(z, w)$ in terms of $\theta$ and $\rho$, 
use (RLY6), 
rewrite $[\delta(x, y), [\rho(z), \rho(w)]]$ by Jacobi identity in $\mathrm{End}(V)$ 
and then apply (RLY5) and (LY5). 
\end{remark}

\begin{definition}\label{def_rism}
A \emph{representation of an infinitesimal $s$-manifold} $(T, \sigma)$ 
is a data $(V, \rho, \delta, \theta, \psi)$ 
where $(V, \rho, \delta, \theta)$ is a representation  of $T$ 
and $\psi\in \mathrm{End}(V)$ is an invertible linear transformation 
satisfying the following for any $x, y\in T$. 
\begin{enumerate}
\item[(RISM1)]
$\rho(\sigma(x))=\psi\circ\rho(x)\circ\psi^{-1}$. 
\item[(RISM2)]
$\theta(x, \sigma(y)) = \psi\circ \theta(x, y)$, \ 
$\theta(\sigma(x), y) = \theta(x, y)\circ\psi^{-1}$. 
\item[(RISM3)]
$\delta(x, y) = \psi\circ \delta(x, y)\circ\psi^{-1}$. 
\end{enumerate}

A representation $(V, \psi)$ is called \emph{regular} if $V$ is finite dimensional and 
$1-\psi$ is invertible. 

A \emph{homomorphism} of representations of an infinitesimal $s$-manifold 
is defined as a homomorphism of representations of Lie-Yamaguti algebras 
commuting with $\psi$. 
\end{definition}

\begin{remark}
We may replace one of the equalities in (RISM2) by 
\begin{enumerate}
\item[(RISM4)]
$\theta(\sigma(x), \sigma(y)) = \psi\circ\theta(x, y)\circ\psi^{-1}$. 
\end{enumerate}
From (RLY1), (RISM1), (RISM4) and (ISM1) we also have 
\begin{enumerate}
\item[(RISM5)]
$\delta(\sigma(x), \sigma(y))=\psi\circ \delta(x, y)\circ\psi^{-1}$. 
\end{enumerate}
\end{remark}

Yamaguti's definition of a representation 
was based on Eilenberg's principle of correspondence between 
representations of algebras 
(defined by binary operations and equational relations) 
and their extensions. 
He proved, following Chevalley-Eilenberg, 
a correspondence between extensions of Lie-Yamaguti algebras by abelian ideals 
and pairs of a representation and an element of a certain cohomology group 
(\cite[\S7]{Yamaguti1969}). 

Here, we give the correspondence 
in the case of split extensions, in a form suited for our purpose, 
as well as a proof for the reader's convenience.

By a Lie-Yamaguti algebra over $T$, 
we will mean a pair of a Lie-Yamaguti algebra $T^*$ 
and a homomorphism $T^*\to T$. 
Note that, if $T_1$ and $T_2$ are Lie-Yamaguti algebras over $T$, 
so is $T_1\times_T T_2$ in a natural way. 

\begin{proposition}\label{prop_ly_ext}
(1)
For a Lie-Yamaguti algebra $T$ and a vector space $V$, 
the following data are equivalent. 
Here, the map $T\oplus V\to T$ will always be the standard projection. 
In (iii), 
$V$ is regarded as a Lie-Yamaguti algebra with trivial products.

\begin{enumerate}
\item[(i)]
A Lie-Yamaguti algebra structure on $T\oplus V$ over $T$
for which $T\subseteq T\oplus V$ is a subalgebra 
and $V$ is an abelian ideal. 
\item[(i)']
A Lie-Yamaguti algebra structure on $T\oplus V$ over $T$
of the form 
\begin{eqnarray*}
(x_1, v_1)*(x_2, v_2) & = & (x_1*x_2, A(x_1)v_2+B(x_2)v_1), \\
{[}(x_1, v_1), (x_2, v_2), (x_3, v_3)] & = & 
([x_1, x_2, x_3], \\
& & \qquad E(x_2, x_3)v_1+F(x_3, x_1)v_2 + G(x_1, x_2)v_3), 
\end{eqnarray*}
where $A, B: T\to\mathrm{End}(V)$ are linear 
and $E, F, G: T\times T\to\mathrm{End}(V)$ are bilinear. 

We may also require that $B=-A, F(y, x)=-E(x, y)$ and $G(x, x)=0$. 
\item[(ii)]
A Lie-Yamaguti algebra structure on $T\oplus V$ over $T$ satisfying the following. 
\begin{enumerate}
\item[(a)]
$\alpha=(id_T, +): T\oplus V\oplus V\to T\oplus V$ is a homomorphism, 
where $T\oplus V\oplus V$ is identified with $(T\oplus V)\times_T (T\oplus V)$. 
\item[(b)]
For any scalar $\lambda$, 
$\mu_\lambda=(id_T, \lambda): T\oplus V\to T\oplus V; (x, v)\mapsto(x, \lambda v)$ 
is a homomorphism. 
\end{enumerate}

\item[(ii)']
A Lie-Yamaguti algebra structure on $T\oplus V$ over $T$ satisfying (ii)(a). 

\item[(iii)]
An isomorphism class of the following data: 
An extension $0\to V\to \tilde{T}\overset{\pi}{\to} T\to 0$ of Lie-Yamaguti algebras,  
Lie-Yamaguti algebra homomorphisms 
$\alpha: \tilde{T}\times_T\tilde{T}\to \tilde{T}$, $\zeta: T\to \tilde{T}$ over $T$ 
and $\iota: \tilde{T}\to\tilde{T}$ 
such that $(\tilde{T}, \pi, \zeta, \alpha, \iota)$ forms an abelian group object over $T$. 
\item[(iv)]
A representation of $T$ on $V$. 
\end{enumerate}
The correspondence of (i)' and (iv) is given by 
$\rho=A$, $\delta=G$ and $\theta=E$. 

(2)
Let $(T, \sigma)$ be an infinitesimal $s$-manifold, 
$(V, \rho, \delta, \theta)$ a representaion of $T$ and $\psi\in\mathrm{GL}(V)$. 
Then $(T\oplus V, \sigma\times \psi)$ is an infinitesimal $s$-manifold 
if and only if 
$(V, \psi)$ is a regular representation of $(T, \sigma)$, 
and if so, the projection map $T\oplus V\to T$, the zero section $T\to T\oplus V$, 
$\alpha$ and $\mu_\lambda$ in (1)(ii) are homomorphisms 
of infinitesimal $s$-manifolds. 
\end{proposition}

\begin{proof}
(1)
We first show that (i)-(ii)' are equivalent. 
From any of these, using the condition that the projection $T\oplus V\to T$ 
preserves the product $*$, 
it follows that we can write 
\[
(x_1, v_1)*(x_2, v_2) =  
(x_1*x_2, A(x_1)v_2+B(x_2)v_1+C_0(x_1, x_2)+C_2(v_1, v_2)), 
\]
where $A$ and $B$ are linear with values in $\mathrm{End}(V)$ 
and $C_0$ and $C_2$ are bilinear with values in $V$. 
Similarly, the first component of 
$[(x_1, v_1), (x_2, v_2), (x_3, v_3)]$ 
is $[x_1, x_2, x_3]$ and the second component can be written as
\begin{multline}\label{formula_tp}
E(x_2, x_3)v_1+F(x_3, x_1)v_2 + G(x_1, x_2)v_3 \\
+ D_0(x_1, x_2, x_3) 
+ D_2^{(1)}(x_1, v_2, v_3) + D_2^{(2)}(v_1, x_2, v_3) + D_2^{(3)}(v_1, v_2, x_3)
+ D_3(v_1, v_2, v_3), 
\end{multline}
where $E, F, G: T\times T\to \mathrm{End}(V)$ are bilinear 
and $D_1$, $D_2^{(j)}$ and $D_3$ are trilinear with values in $V$. 

Let us assume (i). Then $C_0=0$ since $T$ is a closed under $*$  
and $C_2=0$ since $*$ is zero on $V$. 
Similarly, the second component of the triple product 
is also linear in $v_1, v_2, v_3$. 
In fact, from the assumption that $T$ is closed under $[\ ]$, 
it follows that $D_0=0$; 
from $[V, V, V]=0$ we have $D_3=0$; 
from $[T, V, V]=0$ we have $D_2^{(1)}=0$, and so on. 
Thus the operations are as in (i)'. 
From (LY1) for $T\oplus V$, $B=-A$ follows. 
From (LY2) for $T\oplus V$, 
$F(y, x)=-E(x, y)$ and $G(x, x)=0$ follow. 

Conversely, assuming (i)', 
it is easy to see that $T$ is a subalgebra and $V$ is an abelian ideal. 
(For the latter, note that $V$ is an ideal since it is the kernel of $T\oplus V\to T$.) 
Thus (i) and (i)' are equivalent. 

From (i)', it is straightforward to show that (ii) and (ii)' hold. 
For example, we have 
\begin{eqnarray*}
\mu_\lambda(x_1, v_1)*\mu_\lambda(x_2, v_2)
& = &  (x_1, \lambda v_1)*(x_2, \lambda v_2) \\
& = & (x_1*x_2, A(x_1)(\lambda v_2)+B(x_2)(\lambda v_1)) \\
& = & \mu_\lambda ((x_1, v_1)*(x_2, v_2)). 
\end{eqnarray*}

Let us assume (ii)' and show that 
the second components of the products $(x_1, v_1)*(x_2, v_2)$ 
and $[(x_1, v_1), (x_2, v_2), (x_3, v_3)]$ are again linear in $v_i$. 
To show the latter, say, 
we look at the second components of the both sides of the relation 
\[
\alpha([(x_1, v_1, w_1), (x_2, v_2, w_2), (x_3, v_3, w_3)]) 
= [(x_1, v_1+w_1), (x_2, v_2+w_2), (x_3, v_3+w_3)]. 
\]
For the left hand side, it is the sum of the second components of 
the two products 
$[(x_1, v_1), (x_2, v_2), (x_3, v_3)]$ and $[(x_1, w_1), (x_2, w_2), (x_3, w_3)]$. 
By setting $v_i=w_i=0$, we obtain $2D_0(x_1, x_2, x_3)=D_0(x_1, x_2, x_3)$, 
i.e. $D_0\equiv 0$. 
By setting $x_i=0$ and $w_1=v_1, v_2=v_3=0$, we have $D_3\equiv 0$. 
Then $x_2=x_3=0$, $w_2=v_2, w_3=0$ gives $D_2^{(1)}\equiv 0$, 
and so on.  
Thus (i)' holds, 
and we have shown that (i)-(ii)' are equivalent.

Now let us show that (i)-(ii)' are equivalent to (iii). 
Given a Lie-Yamaguti algebra structure on $T\oplus V$ satisfying (ii), 
let $\tilde{T}:=T\oplus V$, $\pi$ the projection, 
$\alpha$ the morphism (ii)(a), $\zeta: T\to T\oplus V; x\mapsto (x, 0)$ 
and $\iota: T\oplus T\to T\oplus V; (x, v)\mapsto (x, -v)$. 
Using the linearity of the operations in $v$, 
it is easy to see that these are Lie-Yamaguti homomorphisms and 
that $(\tilde{T}, \pi, \zeta, \alpha, \iota)$ satisfies the conditions of (iii). 

Conversely, let an extension $\tilde{T}$ and $\pi, \zeta, \alpha, \iota$ as in (iii) be given. 
The section $\zeta$ gives a natural isomorphism 
$f: T\oplus V\to \tilde{T}; (x, v)\mapsto \zeta(x)+v$ as vector spaces. 
Applying $\alpha\circ (\zeta\times_T id_{\tilde{T}})=id_{\tilde{T}}$ 
to $(0, v)\in T\times_T\tilde{T}$, $v\in V$, 
we have $\alpha(0, v)=v$ for $v\in V$. 
By symmetry we have $\alpha(v, 0)=v$. 
Applying the same equality to $(x, \zeta(x))\in T\times_T\tilde{T}$, 
we have $\alpha(\zeta(x), \zeta(x))=\zeta(x)$. 
From the linearity of $\alpha$, we have 
\begin{eqnarray*}
\alpha(f(x, v), f(x, w)) & = & \alpha(\zeta(x)+v, \zeta(x)+w) \\
& = & \alpha(\zeta(x), \zeta(x)) + \alpha(v, 0)+\alpha(0, w) \\
& = & \zeta(x)+v+w = f(x, v+w), 
\end{eqnarray*}
i.e. $\alpha$ in (iii) pulls back to $\alpha$ in (ii)(a). 
By taking the pullback of the Lie-Yamaguti algebra structure, 
we are in the situation of (ii)'. 
Obviously, the result only depends on the isomorphism class of 
the extension, $\zeta$ and $\alpha$. 

Let us show that these correspondences are inverse to each other. 
It is obvious that if we start from (ii) we have $f=id_{T\oplus V}$ and 
obtain the original data. 
Conversely, if we start from (iii), the isomorphism $T\oplus V\cong \tilde{T}$ is 
constructed so that the extension $0\to V\to \tilde{T}\to T\to 0$ 
corresponds to the standard extension $0\to V\to T\oplus V\to T\to 0$ and 
$\alpha$ in (iii) corresponds to $\alpha$ in (ii). 
Thus it suffices to show that $\zeta$ and $\iota$ in (iii) are uniquely determined, 
where $\tilde{T}=T\oplus V$ is constructed from data as in (ii). 

In this case $\zeta: T\to T\oplus V$ is given as $x\mapsto (x, g(x))$. 
Then the group axiom 
$\alpha\circ (\zeta\times_T id_{\tilde{T}})=id_{\tilde{T}}$ 
applied to $(x, \zeta(x))$ gives 
$(x, g(x)+g(x))=(x, g(x))$, i.e. $g\equiv 0$. 
Similarly, 
$\alpha\circ (\iota, id_{\tilde{T}})=id_{\tilde{T}}$ 
applied to $((x,v), \iota(x, v))$ gives $\iota(x, v)=(x, -v)$. 

Equivalence with (iv): 
For the operations in the form of (i)', 
let us show that the following conditions are equivalent: 
\begin{itemize}
\item[(a)]
For $A, B, E, F$ and $G$, 
the associated operations make $T\oplus V$ a Lie-Yamaguti algebra 
(over $T$, satisfying (i)--(ii)'). 
\item[(b)]
$(V, A, G, E)$ is a representation of $T$, 
and $B$ and $F$ are given by 
$B=-A$ and $F(y, x)=-E(x, y)$. 
\end{itemize}

In fact, the condition $(x, a)*(x, a)=0$, 
which is the skew-symmetry condition (LY1) for $T\oplus V$, is 
readily seen to be equivalent to $B=-A$. 
The condition (LY2) for $T\oplus V$, $[(x, a), (x, a), (y, b)]=0$, 
is equivalent to $F(y, x)=-E(x, y)$ and $G(x, x)=0$. 
Note that (b) implies $G(x, x)=0$ by virtue of (RLY1). 

Thus, assuming either (a) or (b), 
we have 
\[
(x, a)*(y, b)= (x*y, A(x)b-A(y)a)
\]
and 
\[
[(x, a), (y, b), (z, c)] = ([x, y, z], E(y, z)a-E(x, z)b+G(x, y)c), 
\]
with $G$ skew-symmetric. 
We will therefore assume these in the rest of the proof. 

By 
\[
((x, a)*(y, b))*(z, c) = ((x*y)*z, A(x*y)c+A(z)A(y)a-A(z)A(x)b), 
\]
the second component of the cyclic sum of 
$[(x, a), (y, b), (z, c)]+((x, a)*(y, b))*(z, c)$ is 
\begin{eqnarray*}
& & \left(G(y, z)+E(y, z)-E(z, y) -[A(y), A(z)]+ A(y*z)\right) a \\
& + & \left(G(z, x)+E(z, x)-E(x, z) -[A(z), A(x)]+ A(z*x)\right) b \\
& + & \left(G(x, y)+E(x, y)-E(y, x) -[A(x), A(y)]+ A(x*y)\right) c. 
\end{eqnarray*}
This is identically $0$ if and only if 
$G(x, y)+E(x, y)-E(y, x)=[A(x), A(y)]-A(x*y)$ holds for any $x$ and $y$. 
Thus we have shown that 
(RLY1) for $(V, A, G, E)$ 
is equivalent to (LY3) for $T\oplus V$. 

The second component of $[(x, a)*(y, b), (z, c), (w, d)]$ is 
\[
-E(z, w)A(y)a+E(z, w)A(x)b-E(x*y, w)c+G(x*y, z)d, 
\]
and if we take the cyclic sum in the first $3$ arguments, 
the coefficient of $a$ is 
\[
-E(y*z, w)-E(z, w)A(y)+E(y, w)A(z), 
\]
those of $b$ and $c$ are cyclic permutation of this, 
and that of $d$ is 
\[
G(x*y, z)+G(y*z, x)+G(z*x, y). 
\]
Thus (LY4) holds for $T\oplus V$ if and only if 
(RLY3) and (RLY7) for $(V, A, G, E)$ holds. 

The second components of $[(x, a), (y, b), (z, c)*(w, d)]$, 
$[(x, a), (y, b), (z, c)]*(w, d)$ and $(z, c)*[(x, a), (y, b), (w, d)]$ are 
\begin{eqnarray*}
& & E(y, z*w)a-E(x, z*w)b+G(x, y)(A(z)d-A(w)c), \\
& & A([x, y, z])d-A(w)(E(y, z)a-E(x, z)b + G(x, y)c), \\
& & A(z)(E(y, w)a-E(x, w)b+G(x, y)d)-A([x, y, w])c. 
\end{eqnarray*}
Thus (LY5) for $T\oplus V$ holds if and only if 
\begin{eqnarray*}
E(y, z*w) & = & -A(w)E(y, z) + A(z)E(y, w), \\
 -E(x, z*w) & = & A(w)E(x, z) - A(z)E(x, w),  \\
-G(x, y)A(w)& = & -A(w)G(x, y) - A([x, y, w]) \hbox{ and } \\
G(x, y)A(z) & = & A([x, y, z]) + A(z)G(x, y) 
\end{eqnarray*}
hold. 
The first two equalities correspond to (RLY2), 
and the other two  to (RLY5). 

Finally, we calculate 
\begin{eqnarray*}
[(x, a), (y, b), [(z, c), (v, d), (w, e)]] 
 & - & [[(x, a), (y, b), (z, c)], (v, d), (w, e)] \\
 & - & [(z, c), [(x, a), (y, b), (v, d)], (w, e)]] \\
 & - & [(z, c), (v, d), [(x, a), (y, b), (w, e)]]. 
\end{eqnarray*}
Then we see that (LY6) for $T\oplus V$ is equivalent to the following: 
\begin{eqnarray*}
E(y, [z, v, w]) - E(v, w)E(y, z) + E(z, w)E(y, v) - G(z, v)E(y, w) & = & 0, \\
-E(x, [z, v, w]) + E(v, w)E(x, z) - E(z, w)E(x, v) + G(z, v)E(x, w) & = & 0, \\
G(x, y)E(v, w) - E(v, w)G(x, y) - E([x, y, v], w) - E(v, [x, y, w]) & = & 0, \\
-G(x, y)E(z, w) + E([x, y, z], w) + E(z, w)G(x, y) + E(z, [x, y, w]) & = & 0, \\
G(x, y)G(z, v) - G([x, y, z], v) - G(z, [x, y, v]) - G(z, v)G(x, y) & = & 0. 
\end{eqnarray*}
The first two equalities are equivalent to (RLY4), 
the third and forth to (RLY6), 
and the last to (RLY8).

(2)
Let $(x, a), (y, b)$ and $(z, c)$ be elements of $T\oplus V$. 
Write $\tilde{\sigma}$ for $\sigma\times \psi$. 
We have
\[
\tilde{\sigma}((x, a)*(y, b))
=
(\sigma(x*y), \psi(\rho(x)b)-\psi(\rho(y)a)), 
\]
while 
\[
\tilde{\sigma}(x, a)*\tilde{\sigma}(y, b)
=
(\sigma(x)*\sigma(y), \rho(\sigma(x))\psi(b)-\rho(\sigma(y))\psi(a)), 
\]
hence (ISM1) for $T\oplus V$ holds if and only if 
(RISM1) holds. 

We have 
\[
\tilde{\sigma}([(x, a), (y, b), (z, c)]) 
=(\sigma([x, y, z]), 
\psi(\theta(y, z)a)-\psi(\theta(x, z)b)+\psi(\delta(x, y)c))
\]
and 
\begin{eqnarray*}
& & [\tilde{\sigma}(x, a), \tilde{\sigma}(y, b), \tilde{\sigma}(z, c)] \\
& = &
([\sigma(x), \sigma(y), \sigma(z)], 
\theta(\sigma(y), \sigma(z))\psi(a)-\theta(\sigma(x), \sigma(z))\psi(b)
+\delta(\sigma(x), \sigma(y))\psi(c)), 
\end{eqnarray*}
hence (ISM2) for $T\oplus V$ holds if and only if 
(RISM4) and (RISM5) hold. 

Finally, we have 
\[
[(x, a), (y, b), \tilde{\sigma}(z, c)] 
= 
([x, y, \sigma(z)], 
\theta(y, \sigma(z))a-\theta(x, \sigma(z))b
+\delta(x, y)\psi(c)), 
\]
hence (ISM3) for $T\oplus V$ is equivalent to 
(RISM3) and the first equality of (RISM2) for $(V, \psi)$. 
Note that the first equality of (RISM2) and (RISM4) imply (RISM2). 
\end{proof}

\begin{proposition}\label{prop_hom_ly_ext}
Let $(T, \sigma)$ be an infinitesimal $s$-manifold 
and 
$(V_1, \psi_1)$ and $(V_2, \psi_2)$ its regular representations. 

(1)
A linear map $f: V_1\to V_2$ is a homomorphism of representations 
if and only if $id_T\times f: T\oplus V_1\to T\oplus V_2$ 
is a homomorphism of infinitesimal $s$-manifolds. 

(2)
If $\tilde{f}: T\oplus V_1\to T\oplus V_2$ is a homomorphism of infinitesimal $s$-manifolds over $T$ 
compatible with $\alpha$, 
then it is given as $\tilde{f}=id_T\times f$ for some $f$. 
\end{proposition}

\begin{proof}
(1) 
This is straightforward from the description of 
the structures of infinitesimal $s$-manifolds on $T\oplus V_1$ and $T\oplus V_2$ 
(Proposition \ref{prop_ly_ext} (1) (i)', (iv), (2)) 

(2)
Since $\tilde{f}$ is a homomorphism over $T$, 
it restricts to $f: V_1\to V_2$. 

If we write $\tilde{f}(x, 0)=(x, g(x))$, 
then $\tilde{f}(\alpha((x, 0), (x, 0)))=\alpha(\tilde{f}(x, 0), \tilde{f}(x, 0))$ 
can be rewritten as $(x, g(x))=(x, g(x)+g(x))$. 
Hence we have $g(x)=0$ and the assertion follows. 
\end{proof}

\begin{corollary}\label{cor_eq_ly_ext}
The correspondence of Proposition \ref{prop_ly_ext} gives an equivalence of the following categories: 
\begin{itemize}
\item[(a)]
The category $\mathbf{Rep}^r(T, \sigma)$ of regular representations 
$(V, \rho, \delta, \theta, \psi)$ of $(T, \sigma)$. 
\item[(b)]
The category $\mathbf{Ab}(\mathbf{Ism}/(T, \sigma))$ of 
abelian group objects in $\mathbf{Ism}$ over $(T, \sigma)$: 
Objects are $((\tilde{T}, \tilde{\sigma}), \pi, \zeta, \alpha, \iota)$, 
where 
\begin{itemize}
\item
$(\tilde{T}, \tilde{\sigma})$ is an infinitesimal $s$-manifold, 
\item
$\pi: \tilde{T}\to T$ is a homomorphism of infinitesimal $s$-manifolds, 
by which $\tilde{T}$ is regarded as an infinitesimal $s$-manifold over $T$, and 
\item
$\zeta: T\to \tilde{T}$, 
$\alpha: \tilde{T}\times_T\tilde{T}\to \tilde{T}$ and $\iota: \tilde{T}\to\tilde{T}$ 
are homomorphisms over $T$, 
\end{itemize}
satisfying the ``axiom of abelian groups'' over $T$. 

The morphisms are homomorphisms of infinitesimal $s$-manifolds over $T$ 
compatible with other data. 
\end{itemize}
\end{corollary}
\begin{proof}
From (a), Proposition \ref{prop_ly_ext} (iv) $\Rightarrow$ (iii) 
gives an object of (b). 

Proposition \ref{prop_ly_ext} (iii) $\Rightarrow$ (iv) 
gives a correspondence in the other direction. 

By the previous proposition, these correspondences give 
an equivalence of categories. 

\end{proof}

\section{Main theorem}

Now we state and prove the main theorem. 
For a vector bundle (or a locally free sheaf) $\mathcal{A}$ over $Q$ and a point $q\in Q$, 
$\mathcal{A}_q$ will denote the fiber at $q$, 
and if $F: \mathcal{A}\to \mathcal{A}'$ is a linear bundle map, 
$F_q$ denotes the linear map $\mathcal{A}_q\to \mathcal{A}'_q$ between the fibers. 

\begin{theorem}\label{thm_main}
Let $Q$ be a connected regular $s$-manifold and $q\in Q$ a point. 
Write $\mathbb{K}$ for the base field. 

(1)
Given a regular quandle module $\mathcal{A}$ over $Q$, 
there is a natural structure of a regular representation of $(T_qQ, d_qs_q)$ on $\mathcal{A}_q$. 

(2)
For regular quandle modules $\mathcal{A}$ and $\mathcal{A}'$ over $Q$ 
and a homomorphism $F: \mathcal{A}\to\mathcal{A}'$, 
$F_q$ is a homomorphism of representations of $(T_qQ, d_qs_q)$. 
This correspondence gives an injective map 
between the sets of homomorphisms. 
In the smooth or complex analytic case, 
this map is surjective if $Q$ is simply-connected. 

(3)
Assume that $Q$ is a connected, simply-connected smooth or complex analytic 
regular $s$-manifold. 
Given a regular representation $V$ of $(T_qQ, d_qs_q)$, 
there exists a regular quandle module $\mathcal{A}$ over $Q$ 
such that $\mathcal{A}_q$ is isomorphic to $V$. 

\smallbreak
In other words, there is a faithful functor between the following categories: 
\begin{itemize}
\item
The category $\mathbf{Mod}_{\mathbb{K}}^r(Q)$ 
of regular quandle modules $(\mathcal{A}, \eta, \tau)$ over $(Q, \rhd)$. 
\item
The category $\mathbf{Rep}^r(T_qQ, d_qs_q)$ 
of regular representations $(V, \rho, \delta, \theta, \psi)$ 
of the infinitesimal $s$-manifold $(T_qQ, *, [\ ], d_qs_q)$. 
\end{itemize}
In the smooth or complex analytic case, 
if $Q$ is simply-connected, it is an equivalence. 
\end{theorem}

\begin{proof}
Let us write $(T, \sigma)$ for $(T_qQ, d_qs_q)$. 

(1)
Recall that $\mathcal{A}$ is a vector bundle. 
By Proposition \ref{prop_quandle_space_from_module}, 
the total space $\mathcal{A}$ 
and $\mathcal{A}\times_Q\mathcal{A}$ 
can be considered as a smooth quandle, etc.,  
such that the projection $\Pi: \mathcal{A}\to Q$, 
the zero section $Z: Q\to\mathcal{A}$, 
the fiberwise addition $A: \mathcal{A}\times_Q \mathcal{A}\to \mathcal{A}$ 
and the fiberwise inversion $I: \mathcal{A}\to\mathcal{A}$ 
are homomorphisms of smooth quandles, etc. 
Note that $\mathcal{A}\times_Q\mathcal{A}$
can also be considered as the quandle 
associated to the direct sum quandle module $\mathcal{A}\oplus \mathcal{A}$. 

We write $(q, a)$ for a point $a\in \mathcal{A}_q$ 
and $(q, a, b)$ for $(a, b)\in (\mathcal{A}\times_Q \mathcal{A})_q$ 
and $V$ for $\mathcal{A}_q$. 
Then we have natural identifications  
\[
T_{(q, 0)}\mathcal{A}\cong T\oplus V, \qquad
T_{(q, 0, 0)}(\mathcal{A}\times_Q \mathcal{A})
\cong T\oplus V\oplus  V
\]
as vector spaces, by identifying $T$ with the tangent spaces 
to the zero sections. 
The linear maps $d_{(q, 0)}s_{(q, 0)}$ and $d_{(q, 0, 0)}s_{(q, 0, 0)}$ 
can be identified with $(d_qs_q, \eta_{qq})$ and $(d_qs_q, \eta_{qq}, \eta_{qq})$. 
By the definition of a regular quandle module, $1-\eta_{qq}$ is invertible, 
and so is $1-\eta_{xx}$ for any $x\in Q$ by the transitivity of $Q$. 
It follows that the action of $\mathcal{A}_x$ on $\mathcal{A}_x$ from the left is transitive. 
Thus the differentials of $d_{(x, a)}s_{(x, a)}$ are conjugates of $d_{(q, 0)}s_{(q, 0)}$ 
for any $a\in \mathcal{A}_x$, 
and $\mathcal{A}$ is a regular $s$-manifold. 
One can prove that $\mathcal{A}\times_Q \mathcal{A}$ is a regular $s$-manifold 
in the same way. 

By Theorem \ref{thm_reg_s_mfd} (3), 
$T_{(q, 0)}\mathcal{A}$ and $T_{(q, 0, 0)}(\mathcal{A}\times_Q \mathcal{A})$ 
are infinitesimal $s$-manifolds in a natural way. 
The latter can be identified with the fiber product 
$T_{(q, 0)}\mathcal{A}\times_T T_{(q, 0)}\mathcal{A}$ 
as an infinitesimal $s$-manifold. 
In fact, the projection maps $P_1, P_2: \mathcal{A}\times_Q \mathcal{A}\to\mathcal{A}$ 
induces such an isomorphism $(d_{(q, 0, 0)}P_1, d_{(q, 0, 0)}P_2)$ of vector spaces, 
and since $P_1$ and $P_2$ are quandle homomorphisms, 
it is a homomorphism of infinitesimal $s$-manifolds 
by Theorem \ref{thm_hom} (1). 

The differential of $\Pi$, $Z$, $A$ and $I$ are 
homomorphisms of infinitesimal $s$-manifolds 
by Theorem \ref{thm_hom} (1), 
and by Corollary \ref{cor_eq_ly_ext} 
we can define a structure of a regular representation of $(T, d_qs_q)$ on $(V, \eta_{qq})$. 
Or, more concretely, 
Proposition \ref{prop_ly_ext} (ii)' $\Rightarrow$ (iv) 
gives the correspondence.

(2)
By Theorem \ref{thm_hom} (1), 
the homomorphism $F$ induces a 
homomorphism $T\oplus \mathcal{A}_q\to T\oplus \mathcal{A}'_q$ 
of infinitesimal $s$-manifolds, 
which is the identity on $T$ 
and maps $\mathcal{A}_q$ to $\mathcal{A}'_q$. 
By Proposition \ref{prop_hom_ly_ext} (1), 
the map $\mathcal{A}_q\to \mathcal{A}'_q$ 
is a homomorphism of representations of $(T, \sigma)$. 

The injectivity of the correspondence follows from Theorem \ref{thm_hom} (1). 

Assume that $Q$ is simply-connected. 
Then $\mathcal{A}$ is also simply-connected. 
By Corollary \ref{cor_eq_ly_ext}, 
a homomorphism $f: \mathcal{A}_q\to \mathcal{A}'_q$ 
of representations of $(T, \sigma)$ 
gives rise to a homomorphism 
$id_T\times f: T\oplus \mathcal{A}_q\to T\oplus \mathcal{A}'_q$ 
of infinitesimal $s$-manifolds 
compatible with the addition maps on the two algebras. 
It is also compatible with $\mu_\lambda$ in Proposition \ref{prop_ly_ext}. 
By Theorem \ref{thm_hom} (2), 
we have a quandle homomorphism $F: \mathcal{A}\to \mathcal{A}'$. 
Since the projection maps, the addition maps and the scalar multiplications 
for $\mathcal{A}$ and $\mathcal{A}'$ 
are quandle homomorphisms whose differentials are 
the projection maps, $\alpha$ and $\mu_\lambda$ 
on $T\oplus \mathcal{A}_q$ and $T\oplus \mathcal{A}'_q$, 
it follows from Theorem \ref{thm_hom} (1) 
that $F$ commutes with the projection maps, the addition maps and scalar multiplication maps. 
Thus it is a homomorphism of linear quandle modules 
by Proposition \ref{prop_quandle_space_from_module}. 

(3)
For a regular representation $(V, \psi)$ of $(T, \sigma)$, 
we define a structure of an infinitesimal $s$-manifold on $T_1:=T\oplus V$ 
according to Proposition \ref{prop_ly_ext}. 
The direct sum $V\oplus V$ gives rise to
an infinitesimal $s$-manifold $T_2:=T\oplus V\oplus V$, 
and from the two projection maps $V\oplus V\to V$ 
we obtain linear maps $p_1, p_2: T_2\to T_1$, 
which are homomorphisms of infinitesimal $s$-manifolds over $T$ 
by Proposition \ref{prop_hom_ly_ext}. 
It follows that the natural identification 
$T_2\cong T_1\times_T T_1$ 
is an isomorphism of infinitesimal $s$-manifolds. 
(This can also be proven by an explicit calculation.) 

Let $\mathcal{A}$ and $\mathcal{A}_2$ 
be connected, simply-connected regular $s$-manifolds 
endowed with points $q_1$ and $q_2$ 
and isomorphisms 
$T_{q_1}\mathcal{A}\cong T_1$ and $T_{q_2}\mathcal{A}_2\cong T_2$, 
which exist by Theorem \ref{thm_reg_s_mfd} (5). 
Later in the proof, 
we will see that $\mathcal{A}_2\cong \mathcal{A}\times_Q\mathcal{A}$. 
We regard $T_1$ and $T_2$ as infinitesimal $s$-manifolds over $T$ 
by the projection maps $\pi: T_1\to T$ and $\pi_2: T_2\to T$. 
By Proposition \ref{prop_ly_ext} (1), (ii) $\Leftrightarrow$ (iii) $\Leftrightarrow$ (iv), 
we have the following homomorphisms over $T$: 
The zero section $\zeta: T\to T_1$, 
the addition $\alpha: T_2\to T_1$, 
the inversion $\iota: T_1\to T_1$ 
and scalar multiplications $\mu_\lambda: T_1\to T_1$ as in Proposition \ref{prop_ly_ext} (1)(ii). 
 By Theorem \ref{thm_hom} (2), 
the projection maps are differentials of 
quandle morphisms $\Pi: \mathcal{A}\to Q$, $\Pi_2: \mathcal{A}_2\to Q$, 
by which we regard $\mathcal{A}$ and $\mathcal{A}_2$ as quandles over $Q$. 
The two projections $p_1, p_2: T_2\to T_1$ 
and $\zeta, \alpha$ and $\mu_\lambda$ 
correspond to homomorphisms 
$P_1, P_2: \mathcal{A}_2\to \mathcal{A}$, 
$Z: Q\to \mathcal{A}$, 
$A: \mathcal{A}_2\to \mathcal{A}$, 
$I: \mathcal{A}\to \mathcal{A}$ 
and $M_\lambda: \mathcal{A}\to \mathcal{A}$ 
of quandles over $Q$. 

Now we will show that $\mathcal{A}$ is a vector bundle over $Q$ 
with the fiber $\mathcal{A}_q := \Pi^{-1}(q)$ isomorphic to $V$. 
Let $G$ and $G_1$ be connected, simply-connected Lie groups whose Lie algebras 
are isomorphic to $\fg(T)$ and $\fg(T_1)$, respectively. 
Since $\pi: T_1\to T$ is surjective, 
we have a Lie algebra homomorphism $\tilde{\pi}: \fg(T_1)\to \fg(T)$ extending $\pi$ 
by Proposition \ref{prop_alg_hom_extended} (1). 
Let $\tilde{\Pi}: G_1\to G$ be the corresponding homomorphism. 
Since $V\subseteq T\oplus V=T_1$ is an abelian ideal, 
$\tilde{V}:=V\oplus \fh(V, T_1)$ is an ideal of $\fg(T_1)$ 
and $V$ is an abelian ideal of $\tilde{V}$ 
by Lemma \ref{lem_subalg_from_ideal}. 
Thus we have connected normal Lie subgroups 
$G_{\tilde{V}}\lhd G_1$ and $G_V\lhd G_{\tilde{V}}$ with 
$T_eG_{\tilde{V}}=\tilde{V}$ and $T_eG_V=V$.

\begin{claim}
The $G_{\tilde{V}}$-orbits in $\mathcal{A}$ are fibers of $\Pi$. 
\end{claim}
\begin{proof}
By Lemma \ref{lem_subalg_from_ideal}(2), 
$G_{\tilde{V}}\subseteq \mathrm{Ker}(\tilde{\Pi})$ holds. 
Thus, for $g\in G_{\tilde{V}}$ and $x\in \mathcal{A}$, 
we have $\Pi(g\cdot x)=\tilde{\Pi}(g)\cdot\Pi(x)=\Pi(x)$, 
i.e. the $G_{\tilde{V}}$-action preserves the fibers of $\Pi$. 

Since $V$ can be identified with the tangent space $T_{q_1}\mathcal{A}_q$ 
of the fiber $\mathcal{A}_q$, 
the orbit $G_{\tilde{V}}\cdot q_1$ contains an open subset of $\mathcal{A}_q$. 
By standard arguments in continuous actions, 
$G_{\tilde{V}}\cdot q_1$ is in fact open in $\mathcal{A}_q$. 
Since $G_{\tilde{V}}$ is normal in $G_1$, 
a translate of a $G_{\tilde{V}}$-orbit by an element of $G_1$ is again a $G_{\tilde{V}}$-orbit. 
By the transitivity of the $G_1$-action on $\mathcal{A}$, 
any $G_{\tilde{V}}$-orbit in $\mathcal{A}$ is a translate of $G_{\tilde{V}}\cdot q_1$, 
and is an open subset of a fiber of $\Pi$. 
It is also closed in the fiber since it is the complement of other orbits, 
and therefore is a connected component of a fiber. 
Since $Q$ is simply-connected and $\mathcal{A}$ is connected, 
fibers of $\Pi$ are connected, 
hence $G_{\tilde{V}}$-orbits are exactly fibers of $\Pi$. 
\end{proof}

If we regard $V$ as a commutative Lie group, 
then $G_V$ is a quotient of $V$ by a discrete subgroup, 
and $V$ acts on $\mathcal{A}$ via $G_V$. 

\begin{claim}
The action of $V$ on $\mathcal{A}_q$ is simply transitive. 
\end{claim}

\begin{proof}
Let $H'$ be the connected Lie subgroup of $G_{\tilde{V}}$ corresponding to $\fh(V, T_1)$. 
We see that $G_{\tilde{V}}$ is generated by $G_V$ and $H'$, 
and $G_{\tilde{V}}=G_V H'$ holds since $G_V$ is normal in $G_{\tilde{V}}$. 

If $H_1$ denotes the connected subgroup of $G_1$ corresponding to $\fh(T_1)$, 
then $\mathcal{A}$ is described as $G_1/H_1$ 
and hence $H_1$ is the stabilizer of $q_1$ in $G_1$. 
From $\fh(V, T_1)\subseteq \fh(T_1)$ 
we see that $H'\subset H_1$ holds. 
Thus $\mathcal{A}_q=G_{\tilde{V}}\cdot q_1=G_V H'\cdot q_1=G_V\cdot q_1$ holds. 
In other words, $\mathcal{A}_q=V\cdot q_1$. 
(We remark that a general $V$-orbit is not necessarily a whole fiber.)

From $V\cong T_{q_1}\mathcal{A}_q$, 
we see that the stabilizer $V_{q_1}$ is discrete. 
If $V_{q_1}\not=0$, 
then we can take an element $v\in V\setminus V_{q_1}$ with $2v\in V_{q_1}$. 
If we consider the map $i: V\to A_q; v\mapsto v\cdot q_1$, 
the map $M_\lambda$ lifts to the usual scalar multiplication on $V$, 
since $i$ is the map 
induced from the Lie algebra homomorphism $V\to \fg(T_1)$, 
for which the scalar multiplication by $\lambda$ on $V$ 
and the map $\mu_\lambda$ on $\fg(T_1)$ are compatible. 
Thus 
\[
i(v)=i\left(\frac{1}{2}\cdot 2v\right)=M_{1/2}(i(2v))=M_{1/2}(q_1)=M_{1/2}(i(0))=
i\left(\frac{1}{2}\cdot 0\right)=q_1, 
\]
a contradiction. 
Thus $V_{q_1}$ is trivial, 
i.e. the action of $V$ on $\mathcal{A}_q$ is simply transitive. 
\end{proof}

\begin{claim}
The map $(P_1, P_2): \mathcal{A}_2\to \mathcal{A}\times_Q \mathcal{A}$ 
is an isomorphism. 
\end{claim}
\begin{proof}
The fiber $\mathcal{A}_q$ is isomorphic to $V$ by the previous claim. 
By the transitivity of $\mathcal{A}$, so is any fiber. 
It follows that any fiber of $\mathcal{A}\times_Q \mathcal{A}$ is 
isomorphic to $V\times V$, 
and that $\mathcal{A}\times_Q \mathcal{A}$ is simply-connected. 

Both $\mathcal{A}_2$ and $\mathcal{A}\times_Q \mathcal{A}$ 
are connected and simply-connected regular $s$-manifolds 
with the infinitesimal algebra $T_1\times_T T_1$, 
so the natural map 
$\mathcal{A}_2\to \mathcal{A}\times_Q \mathcal{A}$ 
is an isomorphism. 
\end{proof}

In particular, $A$ can be considered as 
a binary operation $\mathcal{A}\times_Q \mathcal{A}\to\mathcal{A}$. 
The data $(\mathcal{A}, \Pi, Z, A, I, \{M_\lambda\}_{\lambda\in\mathbb{K}})$ 
satisfies the relative version of axioms of $\mathbb{K}$-vector space over $Q$, 
and in particular defines a structure of a $\mathbb{K}$-vector space 
on each fiber $\mathcal{A}_x$. 

\begin{claim}
Let $g: \mathcal{A}\to\mathcal{A}$ be a homomorphism 
of smooth quandles or complex analytic quandles 
which covers a quandle homomorphism $Q\to Q$ 
and preserves $Z(Q)$. 

Then $g$ induces linear maps between fibers, 
i.e. for any $a, b\in \mathcal{A}_x$ and $\lambda\in \mathbb{K}$, 
we have $g(A(a, b)) = A(g(a), g(b))$ and $g(M_\lambda(a))=M_\lambda(g(a))$. 
\end{claim}
\begin{proof}
Let us prove the first equality. The proof of the second one is similar. 

Write $F_1=g\circ A$ and $F_2=A\circ(g\times_Q g)$. 
They are maps from $\mathcal{A}\times_Q\mathcal{A}$ to $\mathcal{A}$, 
and what we have to show is that they coincide. 
Since they are homomorphism of smooth quandles or complex analytic quandles, 
it is sufficient to show that $F_1(q_2)=F_2(q_2)$ 
and $d_{q_2}F_1=d_{q_2}F_2$ by Theorem \ref{thm_hom} (1). 

Let $Z_2=(Z, Z): Q\to \mathcal{A}\times_Q\mathcal{A}$. 
Then $q_2=(Z(q), Z(q))=Z_2(q)$ 
under the identification $\mathcal{A}_2\cong \mathcal{A}\times_Q\mathcal{A}$. 
Since $g$ covers a map $\bar{g}: Q\to Q$ 
and preserves $Z(Q)$, 
it follows that $g(Z(x))=Z(\bar{g}(x))$ for any $x\in Q$, 
and since $Z(x)$ and $Z(\bar{g}(x))$ are the zero elements of 
$\mathcal{A}_x$ and $\mathcal{A}_{\bar{g}(x)}$ for $A$, respectively, 
we have 
\begin{eqnarray*}
F_1(Z_2(x)) & = & g(A(Z(x), Z(x)))=g(Z(x))=Z(\bar{g}(x)) \hbox{ and} \\
F_2(Z_2(x)) & = & A(g(Z(x)), g(Z(x))) = A(Z(\bar{g}(x)), Z(\bar{g}(x))) = Z(\bar{g}(x)). 
\end{eqnarray*}
Thus $F_1|_{Z_2(Q)}=F_2|_{Z_2(Q)}$ holds.  
In particular, we have $F_1(q_2)=F_2(q_2)$ 
and $(d_{q_2}F_1)_{T_{q_2}Z_2(Q)}=(d_{q_2}F_2)_{T_{q_2}Z_2(Q)}$. 

From the condition that $Z$ is the ``zero element'' for $A$, i.e. 
$A\circ (id_\mathcal{A}\times_Q Z)=A\circ (Z\times_Q id_\mathcal{A})=id_\mathcal{A}$ 
as maps $\mathcal{A}\cong \mathcal{A}\times_Q Q\cong  Q\times_Q \mathcal{A}\to \mathcal{A}$, 
it follows that the restrictions of $d_{q_2}F_1$ and $d_{q_2}F_2$ 
to the first and the second fiber directions 
are all equal to $(d_{q_1}g)|_V: V\to T_{Z(\bar{g}(q))}\mathcal{A}$. 
Thus $d_{q_2}F_1=d_{q_2}F_2$ holds. 
\end{proof}

\begin{claim}
There is a structure of a vector bundle on $\mathcal{A}$ 
for which the projection, the zero section,  the addition map, 
the inversion map and the scalar multiplications are $\Pi$, $Z$, $A$, $I$ 
and $M_\lambda$, respectively. 
\end{claim}
 
\begin{proof}
Let $\tilde{\fg}'$ be the Lie algebra 
obtained by appplying Proposition \ref{prop_alg_hom_extended} to $\zeta: T\to T_1$. 
Then the corresponding Lie group $G'$ maps surjectively to $G$, 
hence the map $a: G'\to Q; g\mapsto g\cdot q$ is surjective, 
and there is a natural map $G'\to G_1$ compatible with $Z$. 

For any point $x_0\in Q$, take a(n \'etale) neighborhood $U$ in $Q$ 
and a local section $s: U\to G'$ for $a$. 
Then the map $\varphi_s: U\times V\to \mathcal{A}; (x, v)\mapsto s(x)(v\cdot q_1)$ gives
an isomorphism to $\Pi^{-1}(U)$ as manifolds over $U$. 

From the explicit form of $\alpha$ and $\mu_\lambda$, 
it is easy to see that the map $v\mapsto v\cdot q_1$ is 
an isomorphism of $V$ and $\mathcal{A}_q$ as linear spaces. 
By the previous claim, 
$\varphi_s$ is linear on each fiber. 
Thus, given two local sections $s, s': U\to G'$ for $a$, 
the transition function $g: U\times V\to V$ defined by 
$\varphi_{s'}^{-1}\circ\varphi_s(x, v)=(x, g(x, v))$ is linear in $v$, 
hence the assertion. 
\end{proof}

Thus $\mathcal{A}$ is a vector bundle with a compatible quandle structure, 
i.e. a quandle module over $Q$ (Proposition \ref{prop_quandle_space_from_module}). 
It is regular since $\eta_{xx}$ are conjugate to $\psi$. 

If we apply (1) to $\mathcal{A}$, 
the isomorphism $T_{q_1}\mathcal{A}\cong T\oplus \mathcal{A}_q$ in (1) 
can be identified with 
the isomorphism $T_{q_1}\mathcal{A} \cong T_eG_1/T_eH_1 \cong T_1=T\oplus V$: 
In the former, the first factor $T$ is given by differentiating $Q\cong Z(Q)\subseteq\mathcal{A}$, 
and in the latter it is the image of $\zeta$, which can be identified with $d_qZ$.
In the former, $\mathcal{A}_q$ is the tangent space to the fiber, 
and so is equal to the kernel of $d_{q_1}\Pi$, and so is equal to $V$ in the latter. 

Therefore the representation of $(T, \sigma)$ associated to $\mathcal{A}$ by (1) 
is isomorphic to $(V, \psi)$.

\end{proof}

\end{document}